\documentclass[11pt]{amsart}
\usepackage{amsmath, amssymb}
\usepackage[right=3cm, left=3cm, top=2cm]{geometry}
\usepackage{enumerate}
\usepackage{nicefrac}

\newcommand{\R}{\mathbb{R}}
\newcommand{\rr}{\R}
\newcommand{\C}{\mathbb{C}}
\newcommand{\cc}{\C}

\newcommand{\fa}{\mathfrak{A}}
\newcommand{\fm}{\mathfrak{M}}
\newcommand{\U}{\widetilde{U}}
\DeclareMathOperator{\re}{Re}
\DeclareMathOperator{\tr}{tr}
\DeclareMathOperator{\rank}{rank}
\DeclareMathOperator{\Ad}{Ad}
\newcommand{\wad}{\widehat{\Ad}}

\newcommand{\pure}{\nu((I - \Lambda)^\frac{1}{2} B (I - \Lambda)^\frac{1}{2})=(f,Bf)}
\newcommand{\cspan}{\mathop{\overline{\mathrm{span}}}}

\newcommand{\hatG}{\widehat{\Gamma}}

\newcommand{\varep}{\varepsilon}
\newtheorem{thm}{Theorem}[section]
\newtheorem{lem}[thm]{Lemma}
\newtheorem{prop}[thm]{Proposition}
\newtheorem{cor}[thm]{Corollary}
\theoremstyle{definition}
\newtheorem{defn}[thm]{Definition}
\newtheorem{remark}[thm]{Remark}
\numberwithin{equation}{section}
\let\oldmarginpar\marginpar
\renewcommand\marginpar[1]{\-\oldmarginpar[\raggedleft\footnotesize #1]%
{\raggedright\footnotesize \textbf{#1}}}

\begin{document}

\title{Gauge groups of $E_0$-semigroups obtained from
Powers weights }
\author{Christopher Jankowski}
\address{Christopher Jankowski, Department of Mathematics,
Ben-Gurion University of the Negev, P.O.B. 653, Beersheva 84105,
Israel.} \email{cjankows@math.bgu.ac.il}
\thanks{The research of the first author was partially supported by the
Skirball Foundation via
the Center for Advanced Studies in Mathematics at Ben-Gurion
University of the Negev. 
}

\author{Daniel Markiewicz}
\address{Daniel Markiewicz, Department of Mathematics,
Ben-Gurion University of the Negev, P.O.B. 653, Beersheva 84105,
Israel.} \email{danielm@math.bgu.ac.il}

\thanks{The research of the second author was partially supported by 
a grant from the U.S.-Israel Binational Science Foundation.
}

\date{February 27, 2011}
\begin{abstract}
The gauge group is computed explicitly for a family of E$_0$-semigroups of type II$_0$ arising from
the boundary weight double construction introduced earlier by Jankowski. This family contains
many E$_0$-semigroups which are not cocycle cocycle conjugate to any examples whose gauge groups have
been computed earlier. Further results are obtained regarding the classification up to cocycle conjugacy
and up to conjugacy for boundary weight doubles $(\phi, \nu)$ in two
separate cases: first in the case when $\phi$ is unital, invertible and $q$-pure and $\nu$ is any type
II Powers weight, and secondly when $\phi$ is a unital $q$-positive map whose range has dimension one and
$\nu(A) = (f, Af)$ for some function $f$ such that  $(1-e^{-x})^{\nicefrac{1}{2}}f(x) \in L^2(0,\infty)$.
All E$_0$-semigroups in the former case are cocycle conjugate to the one arising simply from $\nu$,
and any two E$_0$-semigroups in the latter case  are cocycle conjugate if and only if they are
conjugate.
\end{abstract}
\maketitle

\section{Introduction}
An E$_0$-semigroup is a weak-$*$ continuous one-parameter semigroup of unital $*$-en\-do\-mor\-phisms of a von Neumann algebra $M$. Despite substantial progress in recent years, the classification theory of E$_0$-semigroups up to cocycle conjugacy, which is the appropriate equivalence relation in this context, remains incomplete even in the case when $M=B(H)$,  the algebra of all bounded operators acting on a separable Hilbert space $H$. We recommend the monograph by Arveson~\cite{arv-monograph} as an excellent introduction to the theory of E$_0$-semigroups.

In this paper we will only consider E$_0$-semigroups acting on $B(H)$ with $H$ separable. We will say that an E$_0$-semigroup is spatial if it has a unit, which is a strongly continuous one-parameter semigroup of intertwining isometries. An E$_0$-semigroup is called completely spatial or type I if it is generated by its units in an appropriate sense. A spatial E$_0$-semigroup which is not completely spatial is also called type II, and non-spatial E$_0$-semigroups are called type III. The type of an E$_0$-semigroup is a cocycle conjugacy invariant which turns out to be very coarse. In the spatial case, Powers~\cite{powers1} suggested a finer invariant called the index, which counts the relative abundance of units. Arveson~\cite{arv-analogues} provided a different presentation of the index, proving that it was a well-defined cocycle conjugacy invariant with values in $\{0, 1, 2, \dots, \infty\}$.  He also proved that it completely classifies the E$_0$-semigroups of type I.

No similar classification is known for E$_0$-semigroups of either type II or type III. The existence of such semigroups was established by Powers~\cite{powers-non-spatial, powers-typeII}, and it was later proved by Tsirelson~\cite{tsirelson-holyoke} that there are uncountably many non-cocycle conjugate semigroups in both of those classes. Bhat-Srinivasan~\cite{bhat-typeIII} have further analyzed the examples constructed by Tsirelson, leading to a better understanding of the resulting semigroups accomplished by Izumi~\cite{izumi-perturbations, izumi-sum-systems} and Izumi-Srinivasan~\cite{izumi-srinivasan} in their study of generalized CCR flows and their product systems.

A cocycle conjugacy invariant for E$_0$-semigroups which has attracted attention recently is the gauge group. In the case of an E$_0$-semigroup of type I$_n$ for $n\geq 1$ (the subscript denotes the index), the gauge group was computed by Arveson~\cite{arv-analogues}: for $K$ a Hilbert space with $\dim K=n$,  the gauge group turns out to be isomorphic as a Polish group to the semidirect product of the Heisenberg group $\mathcal{H}^n$ (homeomorphic to  $\rr\times K$ as a topological space)  by the unitary group $U(K)$. In the case when $n=0$, the gauge group is $\rr$. Further progress was made with the introduction of new techniques. Powers~\cite{powers-CPflows} introduced a construction for \emph{all} spatial E$_0$-semigroups by applying Bhat's dilation theorem~\cite{bhat} to a particular kind of semigroup of completely positive maps called CP-flows. Furthermore, Powers proved that his construction ultimately depends on the choice of a single map, called a boundary weight map. The boundary weight map is an adaptation of the resolvent approach for semigroups to the context of CP-flows.  Alevras, Powers and Price~\cite{alevras-powers-price} used the CP-flow approach to describe the gauge group of a large class of E$_0$-semigroups of type II$_0$, namely all such semigroups arising from boundary weight maps over $\cc$. Their description is given with respect to some parameters, however, which are frequently hard to compute.

Further work has been done in the study of transitivity properties of the action of the gauge group on the set of units of a spatial E$_0$-semigroup. Transitivity is automatic for E$_0$-semigroups of type II$_0$, and both transitivity and $2$-fold transitivity (transitivity on the set of pairs of normalized units with fixed covariance) hold for semigroups of type I.  Markiewicz and Powers~\cite{markiewicz-powers} constructed an example of type II$_1$ for which the action need not be $2$-fold transitive.  Concurrently and independently, Tsirelson~\cite{tsirelson-transitivity} constructed an example for which the action need not even be transitive.

In this work we study in detail E$_0$-semigroups of type II$_0$ obtained from the boundary weight double construction introduced by
Jankowski~\cite{jankowski1}. We obtain three different categories of results regarding the classification of such
E$_0$-semigroups,
including the complete and explicit description of the gauge group of a specific subfamily of semigroups whose elements are not cocycle conjugate to any of the E$_0$-semigroups studied by Alevras, Powers and Price~\cite{alevras-powers-price}.

We  now describe the structure and results of the paper. In Section 2, we present the CP-flows approach of Powers and its application to the construction of E$_0$-semigroups from boundary weight doubles. A boundary weight double is a pair $(\phi, \nu)$, where $\phi: M_n(\cc) \to M_n(\cc)$ is a $q$-positive map and $\nu$ is a Powers weight (see Section~\ref{powers-weights-section} for precise definitions).  It was proven in \cite{jankowski1} that if $\phi$ is unital and $\nu$ is a type II Powers weight, then $(\phi, \nu)$ induces an E$_0$-semigroup of type II$_0$.

We proceed in Section 3 to discuss a set of results that are frequently needed in the comparison theory of CP-flows and in the remainder of the paper. Particular cases of these results have been used earlier in applications of the CP-flows approach of Powers.

In Section 4, we prove that if $\phi$ is unital, invertible and $q$-pure, and
if $\nu$ is a type II Powers weight, then the E$_0$-semigroup
induced by $(\phi, \nu)$ is cocycle conjugate to the E$_0$-semigroup induced by $\nu$ in the sense of \cite{powers-holyoke}. Under additional restrictive assumptions on $\nu$, this result was obtained by Jankowski~\cite{jankowski1}.

We then turn in Section 5 to the classification of E$_0$-semigroups arising from boundary weight
 doubles $(\phi, \nu)$ and $(\phi',\nu)$ in the case when $\phi$ and $\phi'$ have range rank one,
 i.e. $\phi(A)=\rho(A) I$ for some state $\rho$, and in addition $\nu(B) = (f, Bf)$ for some function $f$ such that
 $(1-e^{-x})^{\nicefrac{1}{2}}f(x) \in L^2(0,\infty)$. We prove that $(\phi,\nu)$ and $(\phi',\nu)$
 give rise to E$_0$-semigroups which are cocycle conjugate if and only if, in fact, they are
 conjugate.  This is accomplished as an application of a general result regarding the unitary
 equivalence of boundary weight maps and their corresponding CP-flows.

Finally, in the last section of the paper, we compute the gauge group of the E$_0$-semigroups considered in Section 5. Namely, if $\phi(A)=\rho(A)I$ for some state $\rho$ of $M_n(\cc)$ and $\nu(B) = (f, Bf)$ for some  function $f$ such that $(1-e^{-x})^{\nicefrac{1}{2}}f(x) \in L^2(0,\infty)$, then we prove that the gauge group of the E$_0$-semigroup arising from $(\phi, \nu)$ is isomorphic to $\rr \times \rr \times (U_\rho / \mathbb{T})$, where $U_\rho = \{ W \in U(n) : \rho(WAW^*)=\rho(A), \forall A \in M_n(\cc) \}$ and $\mathbb{T} = \{ zI \in U(n) : z\in \cc, |z|=1 \}$.

The authors thank Robert Powers for his helpful observations and comments during the preparation of this article.

\section{Preliminaries}

\subsection{E$_0$-semigroups and CP-flows}

\begin{defn}  Let $H$ be a separable Hilbert space.
We say a family $\alpha = \{\alpha_t\}_{t \geq 0}$ of
normal completely positive contractions of $B(H)$ into itself is a \emph{CP-semigroup} acting on
$B(H)$ if:

(i)  $\alpha_s \circ \alpha_t = \alpha_{s+t}$ for all $s,t \geq 0$
and $\alpha_0 (A)=A$ for all $A \in B(H)$;

(ii) For each $f,g \in H$ and $A \in B(H)$, the inner product
$(f, \alpha_t(A)g)$ is continuous in $t$;

If $\alpha_t(I) =I$ for all $t \geq 0$, then $\alpha$ is called a \emph{unital} CP-semigroup. When
$\alpha$ is a unital CP-semigroup and in addition the map $\alpha_t$ is an endomorphism for every
$t\geq 0$, then $\alpha$ is called an \emph{E$_0$-semigroup}.
\end{defn}

We have two notions of equivalence for $E_0$-semigroups:

\begin{defn}
An $E_0$-semigroup $\alpha$ acting on $B(H_1)$ is \emph{conjugate} to
an E$_0$-semigroup $\beta$ acting on $B(H_2)$  if there exists
a $*$-isomorphism $\theta$ from $B(H_1)$ onto $B(H_2)$ such that
$\theta \circ \alpha_t = \beta_t \circ \theta$ for all $t \geq 0$.

A strongly continuous family of contractions $\mathcal{W}=\{W_t\}_{t \geq 0}$ acting on $H_2$ is
called a contractive $\beta$-cocycle if $W_t \beta_{t}(W_s)=W_{t+s}$ for all $t,s \geq 0$. A
contractive $\beta$-cocycle $W_t$ is said to be a \emph{local cocycle} if for all $A\in B(H_2)$ and
$t\geq 0$, $W_t\beta_t(A) = \beta_t(A)W_t$.

We say $\alpha$ and $\beta$ are \emph{cocycle conjugate} if there exists
a unitary $\beta$-cocycle $\{W_t\}_{t \geq 0}$ such that the E$_0$-semigroup
acting on $B(H_2)$ given by $\beta'_t(A) = W_t \beta_t(A)W_t^*$ for all $A \in B(H_2)$ and $t\geq 0$
is conjugate to $\alpha$.

The set of all local unitary $\beta$-cocycles forms a multiplicative group with respect to the
pointwise operation $(\mathcal{W} \cdot \mathcal{W}')_t=W_tW_t'$. This is called the \emph{gauge
group} of $\beta$ which we denote by $G(\beta)$.
\end{defn}

Let $K$ be a separable Hilbert space. We will always denote by
$\{S_t\}_{t \geq 0}$ the right shift semigroup on  $K \otimes L^2(0,
\infty)$ (which we identify with the space of
$K$-valued measurable functions on $(0, \infty)$ which are square
integrable):
\begin{equation*}
(S_t f)(x) = \left\{ \begin{matrix}
f(x-t), & x> t; \\
0, & x\leq t.
\end{matrix}
\right.
\end{equation*}

\begin{defn}\label{gflow}
A CP-semigroup $\alpha$ acting on $B(K \otimes L^2(0,\infty))$
is called a \textit{CP-flow} over $K$ if $\alpha_t(A)S_t = S_t A$
for all $A \in B(K \otimes L^2(0,\infty))$ and $t \geq 0$.

When $\alpha$ is both a CP-flow and an E$_0$-semigroup, a contractive $\alpha$-cocycle $\{W_t: t\geq
0\}$ is called a \emph{flow cocycle} if $W_tS_t = S_t$ for all $t\geq 0$. We denote by $G_{flow}(\alpha)$ the subgroup of $G(\alpha)$ consisting of all local unitary flow $\alpha$-cocycles.
\end{defn}

We remark that, as a consequence of Theorem 4.61 of \cite{powers-CPflows} (see also the discussion preceding Theorem~1.31 in \cite{alevras-powers-price}), if $\alpha$ is an E$_0$-semigroup of type II$_0$ which is also a CP-flow, then
$$
G(\alpha)=\{ e^{irt} C_t: r \in \rr,\; C \in G_{flow}(\alpha)\}.
$$
Furthermore, the map $(r, C) \mapsto (e^{irt} C_t)_{t\geq 0}$ denotes a canonical isomorphism from the direct product $\rr \times G_{flow}(\alpha)$ onto $G(\alpha)$.

A dilation of a unital CP-semigroup $\alpha$ acting on $B(K)$ is a pair $(\alpha^d, W)$, where
$\alpha^d$ is an E$_0$-semigroup acting on $B(H)$ and $W:K \to H$ is an isometry such that
$\alpha^d_t(WW^*) \geq WW^*$ for $t > 0$ and furthermore
$$
\alpha_t(A) = W^*\alpha_t^d(WAW^*)W
$$
for all $A \in B(K)$ and $t\geq 0$. The dilation is said to be \emph{minimal} if the
span of the vectors
$$
\alpha_{t_1}^d(WA_1W^*)\alpha_{t_2}^d(WA_2W^*)\cdots\alpha_{t_n}^d(WA_nW^*)Wf
$$
for  $f\in K, A_i \in B(K), i=1,\dots n, n\in \mathbb{N}$ is dense in $H$.
This definition of minimality is due to Arveson (see \cite{arv-monograph} for a detailed discussion
regarding dilations of CP-semigroups). We will often suppress the isometry $W$, and refer to a
minimal dilation $\alpha^d$ instead of $(\alpha^d, W)$.

\begin{thm}[Bhat's dilation theorem]\label{thm:Bhat's-dilation-thm}
Every unital CP-semigroup has a minimal dilation which is unique up to conjugacy.
\end{thm}

The following addendum by Powers (Lemma 4.50 of \cite{powers-CPflows}) further clarifies the situation for CP-flows.

\begin{thm}
Every unital CP-flow $\alpha$ has a minimal dilation $\alpha^d$ which is also a CP-flow. We call $\alpha^d$ the minimal flow dilation of the unital CP-flow.
\end{thm}

Given two CP-flows $\alpha$ and $\beta$ over $K$, we will say that $\alpha$ \emph{dominates} $\beta$
or that $\beta$ is a \emph{subordinate} of $\alpha$ if for all $t\geq 0$, the map $\alpha_t
-\beta_t$ is completely positive. We will often denote this relationship by $\alpha \geq \beta$.
Powers~\cite{powers-CPflows} has described a useful criterion for determining whether two CP-flows
have minimal dilations that are cocycle conjugate in terms of the next definition.

\begin{defn}\label{def-corners}
Let $\alpha$ and $\beta$ be CP-flows over $K_1$ and $K_2$, respectively. For $j=1,2$, let
$H_j=K_j \otimes L^2(0,\infty)$ and let $S_t^{(j)}$ denote the right shift on $H_j$. Let $\gamma=\{\gamma_t: t\geq 0\}$ be a family of maps from $B(H_2, H_1)$ into itself and define for each $t>0$, $\gamma_t^*: B(H_1, H_2) \to B(H_1, H_2)$ by $\gamma_t^*(C)=[\gamma_t(C^*)]^*$ for all $C\in B(H_1, H_2)$. We say that $\gamma$ is a \emph{flow corner} from
$\alpha$ to $\beta$ if the maps
$$
\Theta_t \begin{bmatrix}
A & B \\
C & D
\end{bmatrix} =
\begin{bmatrix}
\alpha_t(A) & \gamma_t(B) \\
\gamma_t^*(C) & \beta_t(D)
\end{bmatrix}
$$
define a CP-flow $\Theta=\{ \Theta_t : t\geq 0\}$ over $K_1 \oplus K_2$ with respect to the shift
$S_t^{(1)}\oplus S_t^{(2)}$. Note that $\gamma$ is a flow corner from $\alpha$ to $\beta$ if and only if $\gamma^*$ is a flow corner from $\beta$ to $\alpha$.

A flow corner $\gamma$ is called a \emph{hyper-maximal
flow corner} if every subordinate CP-flow $\Theta'$ of $\Theta$ of the form
$$
\Theta_t' \begin{bmatrix}
A & B \\
C & D
\end{bmatrix} =
\begin{bmatrix}
\alpha_t'(A) & \gamma_t(B) \\
\gamma_t^*(C) & \beta_t'(D)
\end{bmatrix}
$$
for $t\geq 0$ must satisfy $\alpha_t'=\alpha_t$ and $\beta_t'=\beta_t$ for all $t\geq 0$.

More generally, if $\alpha$ is a CP-flow over $K$ and $n$ is a positive integer, we say that
$\Theta$ is a \emph{positive $n\times n$ matrix of flow corners from $\alpha$ to $\alpha$} if
$\Theta=(\theta^{(ij)})$ is a CP-flow over $\oplus_{j=1}^n K$ such that $\theta^{(ii)}$ is a subordinate of $\alpha$ for all $i=1,\dots, n$.
\end{defn}

We also have a notion of $n \times n$ matrices of local flow cocycles (Definition~4.58 of
\cite{powers-CPflows}):

\begin{defn}
Suppose $\alpha$ is a CP-flow which is also an $E_0$-semigroup,
and let $n \in \mathbb{N}$.  We say $C$ is a positive $n \times n$
matrix of local flow $\alpha$-cocycles if the coefficients $C_{ij}$ of $C$ are local flow
$\alpha^d$
cocycles for $i,j=1, \ldots, n$ and the matrix $C(t)$ whose entries are $C_{ij}(t)$
is positive for all $t \geq 0$.
\end{defn}

The following is a combination of Theorems~4.56 and 4.59 in \cite{powers-CPflows}.
\begin{thm}\label{thm-corners}
Suppose $\alpha$ and $\beta$ are unital CP-flows over $K_1$ and $K_2$, respectively, and let
$\alpha^d$ and $\beta^d$ be corresponding minimal flow dilations. If there exists a hyper-maximal
flow corner from $\alpha$ to $\beta$, then $\alpha^d$ and $\beta^d$ are cocycle conjugate.
Conversely, if $\alpha^d$ and $\beta^d$ are cocycle conjugate and in addition $\alpha^d$ is of  type
II$_0$, then there exists a hyper-maximal flow corner from $\alpha$ to $\beta$.

Furthermore, let $(\alpha^d, W)$ be a minimal flow dilation over $H$, so  that
$$
\alpha_t(A) = W^*\alpha_t^d(WAW^*)W
$$
for $A \in B(K_1 \otimes L^2(0,\infty))$ and $t \geq 0$.  Suppose $n$ is a positive integer and
suppose that $\Theta=(\theta^{(ij)})$ is a positive $n\times n$ matrix of flow corners from $\alpha$ to $\alpha$. Then there exists a unique positive $n\times n$ matrix $C=(C_{ij})$ of contractive local flow $\alpha^d$-cocycles such that
\begin{equation}\label{tempcorners}
\theta_t^{(ij)}(A) = W^*C_{ij}(t)\alpha^d_t(WAW^*)W
\end{equation}
for all $A\in B(K_1 \otimes L^2(0,\infty))$. Conversely, if
$C=(C_{ij})$ is a positive $n\times n$ matrix of contractive local flow $\alpha^d$-cocycles,
then the matrix family $\Theta_t$ whose coefficients are given by \eqref{tempcorners} is a positive $n\times n$ matrix of flow corners from $\alpha$ to $\alpha$.
\end{thm}

Theorem~4.60 of \cite{powers-CPflows} tells us when a given flow corner from $\alpha$ to $\alpha$ corresponds to unitary local
$\alpha^d$-cocycle:
\begin{thm}
Suppose $\alpha$ is a unital CP-flow over K and let $\alpha^d$ be a minimal flow dilation.
Suppose $\theta$ is a flow corner from $\alpha$ to $\alpha$ and C is the local contractive flow cocycle
for $\alpha^d$ associated with $\theta$.
Then $C(t)$ is unitary for all $t\geq 0$ if and only if $\theta$
is hyper-maximal.
\end{thm}

\subsection{Boundary weight maps} For the remainder of this section, let $K$ be a fixed separable Hilbert space
(not necessarily infinite-dimensional) and let $H=K \otimes L^2(0,\infty)$.

Define $\Lambda: B(K) \rightarrow B(H)$ by
\begin{equation*}
(\Lambda(A)f)(x)=e^{-x} Af(x)
\end{equation*}
and let $\mathfrak{A}(H)$ be the algebra
\begin{equation*}
\mathfrak{A}(H) = [I - \Lambda(I_K)]^\frac{1}{2} B(H)  [I -
\Lambda(I_K)]^\frac{1}{2}.
\end{equation*}
We will frequently denote by $\Lambda \in B(L^2(0,\infty))$ the operator $\Lambda(I_\cc)$.

\begin{defn}
We say that a linear functional $\mu:\mathfrak{A}(H) \to \C$
is a \textit{boundary weight}, denoted $\mu \in \mathfrak{A}(H)_*$, if the functional
$\ell$ defined on $B(H)$ by
\begin{equation*} \ell(A)= \mu\Big(
[I - \Lambda(I_K)]^\frac{1}{2} A [I - \Lambda(I_K)]^\frac{1}{2} \Big)
\end{equation*}
is a normal bounded linear functional.  The boundary weight $\mu$ is called \emph{bounded} if there exists $C>0$ such that $|\mu(T)| \leq C \|T\| $ for all $T \in \mathfrak{A}(H)$. Otherwise, $\mu$ is called \emph{unbounded}.

A linear from $B(K)_*$ to $\mathfrak{A}(H)_*$ will be called a \emph{boundary weight map}.
\end{defn}
Boundary weights were first defined in \cite{powers-CPflows} (Definition 4.16), where
their relationship to CP-flows was explored in depth.  For an additional discussion of boundary
weights and their properties, we refer
the reader to Definition 1.10 of \cite{markiewicz-powers} and its subsequent remarks.

Given a normal map $\phi: B(H) \to B(K)$, we will denote by $\hat{\phi}:B(K)_* \to B(H)_*$ the
predual map satisfying $\rho(\phi(A)) = (\hat{\phi}(\rho))(A)$ for all $A \in B(H)$ and $\rho \in B(K)_*$.

Define $\Gamma : B(H) \to B(H)$ by the weak* integral
\begin{equation}\label{gamma}
\Gamma(A) = \int_0^\infty e^{-t} S_t A S_t^* dt.
\end{equation}
We record in the following proposition facts which are implicit in the proof of Theorem~4.17 in Powers~\cite{powers-CPflows}, and we present a proof here for the convenience of the reader.

\begin{prop}\label{get-omega}
Let $\mu \in\fa(H)_*$ be a boundary weight. We have that for all $T\in\fa(H)$,
$$
\mu(T) = \lim_{x\to 0+} \mu(S_xS_x^*TS_xS_x^*).
$$
In particular $\mu=\mu'$ if and only if for all $x>0$ and $T \in S_xS_x^*B(H)S_xS_x^*$, we have that $\mu(T) = \mu'(T)$. Furthermore, given $x>0$ and $T \in S_xS_x^*B(H)S_xS_x^*$,
\begin{equation}\label{mut}
\mu(T) = \lim_{y\to x+} \frac{1}{y-x} \hatG(\mu) \Big(T - e^{x-y} S_{y-x} T S_{y-x}^*\Big).
\end{equation}
\end{prop}
\begin{proof} Let $\mu \in\fa(H)_*$ be a boundary weight and let $\ell \in B(H)_*$ be the normal bounded linear functional such that for all $Z\in B(H)$,
$$
\mu((I-\Lambda(I_K))^{1/2}Z(I-\Lambda(I_K))^{1/2}) = \ell(Z).
$$
Given any $T \in \fa(H)$, let $Z \in B(H)$ be such that $T=(I-\Lambda(I_K))^{1/2}Z(I-\Lambda(I_K))^{1/2}$.
Now observe that
\begin{align*}
\mu(T) & = \ell(Z) = \lim_{x \to 0+} \ell(S_xS_x^*ZS_xS_x^*) =  \lim_{x \to 0+} \mu\Big((I-\Lambda(I_K))^{1/2}S_xS_x^*ZS_xS_x^*(I-\Lambda(I_K))^{1/2}\Big) \\
& = \lim_{x \to 0+} \mu\Big(S_xS_x^*(I-\Lambda(I_K))^{1/2}Z(I-\Lambda(I_K))^{1/2}S_xS_x^*\Big) \\
& = \lim_{x \to 0+} \mu(S_xS_x^*TS_xS_x^*).
\end{align*}
It follows immediately from this identity that two boundary weights $\mu, \mu'$ are identical if and only if for every $x>0$ they coincide on the algebra $S_xS_x^*B(H)S_xS_x^*$.

Let $A\in B(H)$ and $x>0$. Observe that if $(A_\lambda)$ is a bounded net of operators in $S_xB(H)S_x^*$ such that
$A_\lambda$ converges ultra-weakly to $S_xAS_x^*$,  then  $\lim_{\lambda} \mu(A_\lambda) =  \mu(S_xA S_x^*)$ (note that $S_x B(H) S_x^* \subseteq\fa(H)$). Indeed, $Q_x = (I-\Lambda(I_K))^{-1/2}S_xS_x^*$ is a bounded operator in $B(H)$ in the natural sense, hence the net $Q_xA_\lambda Q_x^*$ is bounded and also converges ultra-weakly to $Q_xS_xA S_x^*Q_x^*$. Therefore,
\begin{equation}\label{convnull}
\mu(S_xA S_x^*)= \ell(Q_xS_xA S_x^*Q_x^*)
= \lim_{\lambda} \ell(Q_xA_\lambda Q_x^*)
= \lim_{\lambda} \mu(A_\lambda).
\end{equation}

Let $x>0$ be fixed. A straightforward computation shows that for every $A \in B(H)$, $y>x$,
$$
\Gamma\Big( e^{-x} S_xAS_x^* - e^{-y} S_y A S_y^*\Big) = \int_x^y  e^{-t} S_tAS_t^* dt.
$$
The operator on the right obviously belongs to $S_xB(H)S_x^*$ for $y>x$. Furthermore, It is clear that
$$
A_y = \frac{1}{y-x} \int_x^y  e^{-t} S_tAS_t^* dt
$$
is a  bounded net of operators that converges ultra-weakly to $e^{-x} S_xAS_x^*$ as $y \to x$.  Thus for any boundary weight $\mu$, it follows from \eqref{convnull} that
\begin{align*}
\mu(e^{-x} S_xAS_x^*) & = \lim_{y \to x+} \mu \left( \frac{1}{y-x} \int_x^y  e^{-t} S_tAS_t^* dt  \right) \\
& = \lim_{y \to x+} \frac{1}{y-x} \widehat{\Gamma}(\mu) \left(   e^{-x} S_xAS_x^* - e^{-y} S_y A S_y^*  \right).
\end{align*}
Finally, we observe that for every $T \in S_xS_x^*B(H)S_xS_x^*$,  we have that $T=S_xAS_x^*$ for the operator $A=S_x^*TS_x$, hence we obtain equation \eqref{mut} by substitution.
\end{proof}

If $\alpha$ is a CP-flow over $K$, we define its resolvent by the weak* integral
\begin{equation} \label{resolvent-pure}
R_\alpha(A) = \int_0^\infty e^{-t} \alpha_t(A) dt
\end{equation}
defined for $A\in B(H)$.  Powers~\cite{powers-CPflows} proved that there exists a completely positive boundary weight map $\omega:  B(K)_* \to \mathfrak{A}(H)_*$ such that
\begin{equation} \label{resolvent}
\hat{R}_{\alpha}(\eta) = \hat{\Gamma} (\omega(\hat{\Lambda}\eta) + \eta)
\end{equation}
and $\omega(\rho)(I-\Lambda(I_K)) \leq \rho(I_K)$ for all $\rho \in B(K)_*$ positive. Such a boundary weight map is uniquely determined by \eqref{resolvent} in combination with Proposition~\ref{get-omega}, and in fact for all $\rho \in B(K)_*$, $x>0$ and $T \in S_xS_x^*B(H)S_xS_x^*$,
\begin{equation}\label{got-omega}
\omega(\rho)(T) = \lim_{y \to x+} \frac{1}{y-x}
(\widehat{R}_\alpha - \widehat{\Gamma})(\eta) ( T - e^{x-y} S_{y-x}TS_{y-x}^*),
\end{equation}
where $\eta \in B(H)_*$ is any normal functional such that $\rho=\widehat{\Lambda}(\eta)$. Such a functional exists since $\Lambda$ is isometric hence $\widehat{\Lambda}$ is onto.

The map $\omega$ is called \emph{the boundary weight map associated to $\alpha$}.

The following result, which is a compilation of Theorems 4.17, 4.23, and 4.27 of
\cite{powers-CPflows}, describes the converse relationship between boundary weight maps and
CP-flows.

\begin{thm}\label{powerstheorem}
Let $\omega: B(K)_* \to \mathfrak{A}(H)_*$ be a completely positive map satisfying
$\omega(\rho)(I-\Lambda(I_K)) \leq \rho(I_K)$ for all positive
$\rho$.  Let $\{S_t\}_{t \geq 0}$ be the right shift semigroup acting on $H$.
For each $t>0$, define the truncated boundary weight map
$\omega_t: B(K)_* \to B(H)_*$ by
\begin{equation*}
\omega_t(\rho)(A)= \omega(\rho)(S_tS_t^* AS_tS_t^*)
\end{equation*}
If for every $t>0$, the map $(I +\hat{\Lambda}\omega_t)$ is invertible  and furthermore the map
\begin{equation*}
\hat{\pi}_t : = \omega_t(I + \hat{\Lambda}\omega_t)^{-1}
\end{equation*}
is a completely positive contraction from $B(K)_*$ into $B(H)_*$, then $\omega$ is the boundary weight map associated to a CP-flow over $K$.
The CP-flow is unital if and only if $\omega(\rho)(I-\Lambda(I_K))
= \rho(I_K)$ for all $\rho \in B(K)_*$.
\end{thm}

We note that it follows immediately from Proposition~\ref{get-omega} that if $\omega, \omega'$ are two boundary weight maps from $B(K)_*$ to $\fa(H)_*$, then $\omega=\omega'$ if and only if  $\omega_t=\omega_t'$ for all $t>0$.

\begin{defn}
Let $\omega: B(K)_* \to \mathfrak{A}(H)_*$  be a completely positive boundary weight map
satisfying
$\omega(\rho)(I-\Lambda(I_K)) \leq \rho(I_K)$ for all positive
$\rho$. If for every $t>0$ the map $\hat{\pi}_t$ as defined in the statement of Theorem~\ref{powerstheorem} exists and it is a completely positive contraction, then $\omega$ is called a \emph{$q$-positive} boundary weight map. In that case, the
family $\pi_t$ (for $t>0$) of completely positive normal contractions from $B(H)$ to $B(K)$ is called the generalized boundary representation associated to $\omega$, or alternatively to the CP-flow
associated to $\omega$.
\end{defn}

In the next result proven by
Powers~\cite{powers-CPflows} we recall the criterion for subordination in terms of the generalized
boundary representation.

\begin{thm}\label{boundary-representation-subordinates}
Let $\alpha$ and $\alpha'$ be CP-flows acting on $B(H)$ with generalized boundary representations
$\pi_t$ and $\pi'_t$, respectively. Then $\alpha \geq \alpha'$ if and only if $\pi_t-\pi'_t$ is
completely positive for all $t>0$. In particular, if $\pi_t = \pi_t'$ for all $t>0$, then $\alpha=\alpha'$.
\end{thm}

\subsection{Powers weights and boundary weight doubles}\label{powers-weights-section}
A boundary weight map
$\omega: B(\cc)_* \to \mathfrak{A}(L^2(0,\infty))$ is determined by its value $\omega_1:=\omega(1)$, and it  induces a CP-flow $\alpha$ over $\cc$ if and only  $\omega_1$ is a positive boundary weight and $\omega_1(I-\Lambda)\leq1$. In that case,
the CP-flow $\alpha$ is unital if and only if $\omega_1(I-\Lambda)=1$, and
therefore dilates to an $E_0$-semigroup $\alpha^d$.

\emph{Since all the key properties of $\omega$ are determined by the single boundary weight $\omega_1$ in the special case $K=\cc$, we will write $\omega$ instead of $\omega_1$.}

 Results from
\cite{powers-CPflows} show that $\alpha^d$ is of type I if $\omega_1$
is bounded and of type II$_0$ if $\omega_1$ is unbounded.  Thus we are led to the following definition.

\begin{defn} \label{powers-weight} A boundary weight $\nu \in \mathfrak{A}(L^2(0, \infty))_*$ is called
a \emph{Powers
weight} if $\nu$ is positive and $\nu(I-\Lambda)=1$. We say that a Powers weight $\nu$ is
\emph{type I}
if it is bounded and \emph{type II} if it is unbounded.
\end{defn}

If $\nu$ is a Powers weight, then it has the form:
\begin{equation*}
\nu
\Big((I - \Lambda)^\frac{1}{2}A
(I - \Lambda)^\frac{1}{2}\Big) = \sum_{i=1}^k (f_i, Af_i) \end{equation*} for some mutually
orthogonal nonzero $L^2$-functions $\{f_i\}_{i=1}^k$ ($k \in \mathbb{N} \cup
\{\infty\}$) with $\sum_{i=1}^k ||f_i||^2=1$.

We note that if $\nu$ is a
type II Powers weight, then for the weights $\nu_t$ defined by
$\nu_t(A)=\nu(S_tS_t^*AS_tS_t^*)$ for $A \in B(L^2(0, \infty))$
and $t>0$, both $\nu_t(I)$ and $\nu_t(\Lambda)$
approach infinity as $t \rightarrow 0+$.

Powers~\cite{powers-holyoke} has described a useful criterion to determine when Powers weights induce cocycle conjugate E$_0$-semigroups.

\begin{defn}
Let $\nu, \eta \in \mathfrak{A}(L^2(0, \infty))_*$ be positive boundary weights.  We say that $\nu$ \emph{$q$-dominates} $\eta$
(or that $\eta$ is \emph{$q$-subordinate} to $\nu$), and write
$\nu \geq_q \eta$, if
$$
\dfrac{\nu_t}{1+\nu_t(\Lambda)} - \dfrac{\eta_t}{1+\eta_t(\Lambda)}
$$
is a positive element of $B(L^2(0, \infty))_*$ for every $t>0$.

Suppose that $\nu$ and $\eta$ are Powers weights.
We say that $\gamma \in \mathfrak{A}(L^2(0,\infty))_*$ is a \emph{corner} from $\nu$ to $\eta$ if the map from $M_2(\mathfrak{A}(L^2(0,\infty)))$ to $M_2(\cc)$ given by
$$
(A_{ij}) \mapsto \begin{pmatrix}
\nu(A_{11}) & \gamma(A_{12}) \\
\gamma^*(A_{21}) & \eta(A_{22})
\end{pmatrix}
$$
is completely positive. We say that $\gamma$ is a \emph{$q$-corner} from $\nu$ to $\eta$ if for every $t>0$ the map from $M_2(\mathfrak{A}(L^2(0,\infty)))$ to $M_2(\cc)$ given by
$$
(A_{ij}) \mapsto \begin{pmatrix}
\dfrac{\nu_t(A_{11})}{1+\nu_t(\Lambda)} & \dfrac{\gamma_t(A_{12})}{1+\gamma_t(\Lambda)} \\ \noalign{\medskip}
\dfrac{\gamma_t^*(A_{21})}{1+\gamma^*_t(\Lambda)} & \dfrac{\eta_t(A_{22})}{1+\eta_t(\Lambda)}
\end{pmatrix}
$$
is completely positive.

A $q$-corner $\gamma$ is a \emph{hyper-maximal $q$-corner} from $\nu$ to $\eta$ if, whenever
$\nu'$ and $\eta'$ are $q$-subordinates of $\nu$ and $\eta$ such that the map
$$
(A_{ij}) \mapsto \begin{pmatrix}
\dfrac{\nu'_t(A_{11})}{1+\nu'_t(\Lambda)} & \dfrac{\gamma_t(A_{12})}{1+\gamma_t(\Lambda)} \\ \noalign{\medskip}
\dfrac{\gamma_t^*(A_{21})}{1+\gamma^*_t(\Lambda)} & \dfrac{\eta'_t(A_{22})}{1+\eta'_t(\Lambda)}

\end{pmatrix}
$$
is completely positive for each $t>0$, we have $\eta=\eta'$ and $\nu=\nu'$.
\end{defn}

If $\nu$ and $\eta$ are type II Powers weights which induce CP-flows $\alpha$ and
$\beta$, respectively, then there is a bijective correspondence between hyper-maximal
$q$-corners from $\nu$ to $\eta$ and hyper-maximal flow corners from $\alpha$ to $\beta$
(see the discussion preceding Theorem 1.30 of \cite{alevras-powers-price}), whereby
Theorem \ref{thm-corners} implies the following.

\begin{thm}
Let $\nu$ and $\eta$ be type II Powers weights with corresponding CP-flows $\alpha$ and $\beta$,
respectively.  Then $\alpha^d$ and $\beta^d$ are cocycle conjugate if and only if there is a hyper-maximal
$q$-corner from $\nu$ to $\eta$.
\end{thm}

The following theorem describes the set of $q$-corners from a type II Powers weight $\nu$ to itself
(see Definition 2.12 (b) of \cite{alevras-powers-price} and its subsequent discussion):
\begin{thm}
Let $\nu$ be a type II Powers weight, and let $T$ be the trace density
operator associated to $\nu$ in the sense that
$$\nu\Big((I - \Lambda)^\frac{1}{2} A (I - \Lambda)^\frac{1}{2}\Big) = tr(A T)$$
for all $A \in B(L^2(0,\infty))$.
Let $\mathfrak{M}$ be the closure of the range of $T$. For every contraction
$X \in B(\mathfrak{M})$, let $\kappa(X) \in [0,\infty]$ be given by
$$
\kappa(X)=\sup\{\re(\tr(\Lambda(I-\Lambda)^{-1} S_tS_t^* T^\frac{1}{2}(I-X)
T^\frac{1}{2})): t >0\}
$$
Then for every $X \in B(\mathfrak{M})$ such that $\kappa(X) < \infty$
and $x\in\cc$ such that $\re(x) \geq \kappa(X)$, the map
$$
\gamma_{(x,X)}\Big((I - \Lambda)^\frac{1}{2} A (I - \Lambda)^\frac{1}{2}\Big) = \frac{1}{1+x}
\tr(A T^\frac{1}{2} X T^\frac{1}{2})
$$
constitutes a $q$-corner from $\nu$ to $\nu$. Conversely, if $\gamma$ is a $q$-corner from $\nu$ to $\nu$, there exists a unique pair $(x,X)$ such that $X \in B(\mathfrak{M})$ satisfies $\kappa(X) <\infty$ and $x\in\cc$ satisfies $\re(x)\geq\kappa(X)$ such that $\gamma = \gamma_{(x,X)}$.

Furthermore, a $q$-corner $\gamma_{(x,X)}$ is hyper-maximal if and only if $\re(x)=\kappa(X)$ and $X$ is
unitary.
\end{thm}

\begin{remark} \label{hypnote}
Setting $X=I_{\mathfrak{M}}$, we observe that $\kappa(I_\mathfrak{M})=0$, so
if $\re(x) \geq 0$, then the pair $(x,I_\mathfrak{M})$ satisfies the conditions of the theorem.
In other words, if $\re(x) \geq 0$, then $\frac{1}{1+x} \nu$ is a $q$-corner from $\nu$ to $\nu$,
and it is hyper-maximal if and only if $\re(x)=0$.
\end{remark}

We will be interested in combining Powers weights with the  completely positive maps on matrices of the following type to obtain $E_0$-semigroups.
\begin{defn}\label{qposdef} Let $K$ be a separable Hilbert space.
Let $\phi: B(K) \to B(K)$ be a bounded normal linear map with
spectrum contained in $\cc \setminus \{\lambda : \lambda <0 \}$.  We say $\phi$ is \emph{$q$-positive}, and write $\phi \geq_q 0$,
if $\phi(I + t \phi)^{-1}$ is completely positive for all $t \geq 0$.
\end{defn}

We make two observations in light of Definition \ref{qposdef}.
First, it is
not uncommon for a completely positive map to have negative
eigenvalues.  Second, there is no ``slowest rate of failure'' for
$q$-positivity:  For every $s \geq 0$, there exists a linear map
$\phi$ with no negative eigenvalues such that $\phi(I + t
\phi)^{-1}$ ($t \geq 0$) is completely positive if and only if $t \leq s$. These
observations are discussed in detail in section 2.1 of \cite{jankowski3}.

There is a
natural order structure for $q$-positive maps.  If $\phi, \psi:
B(K) \rightarrow B(K)$ are $q$-positive, we say $\phi$
\emph{$q$-dominates} $\psi$ (i.e. $\phi \geq_q \psi$) if $\phi(I + t
\phi)^{-1} - \psi(I + t \psi)^{-1}$ is completely positive for all
$t \geq 0$.  It is not always true that $\phi \geq_q \lambda \phi$ if $\lambda \in (0,1)$
(for a large family of counterexamples, see Theorem~\ref{inverts} below).
However, if $\phi$ is $q$-positive, then for every $s
\geq 0$, we have $\phi \geq_q \phi(I + s \phi)^{-1} \geq_q 0$
(Proposition 4.1 of \cite{jankowski1}).   If these are the only nonzero
$q$-subordinates of $\phi$, we say $\phi$ is \emph{$q$-pure}.

In this paper we will restrict our attention to unital $q$-positive maps over $B(K)$ for $K$ finite-dimensional, and we will approach the case $\dim K=\infty$ in the future.

We have the following result which combines Proposition
3.2 and Corollary 3.3 of \cite{jankowski1}.

\begin{prop} \label{bdryweight} Let $H =\C^n \otimes L^2(0, \infty)$.
Let $\phi: M_n(\C) \rightarrow M_n(\C)$ be a unital $q$-positive map, and let $\nu$ be a type
II Powers weight.  Let $\Omega_\nu: \mathfrak{A}(H) \rightarrow M_n(\C)$
be the map that sends $A=(A_{ij}) \in M_n(\mathfrak{A}(L^2(0, \infty))) \cong \mathfrak{A}(H)$
to the matrix $(\nu(A_{ij})) \in M_n(\C)$.  The map
$\omega: M_n(\C)_* \to
\mathfrak{A}(H)_*$ defined by
\begin{equation*}\omega (\rho) (A) = \rho\Big(\phi(\Omega_\nu(A))\Big), \qquad
\forall A \in \fa(H), \forall\rho \in B(K)_*
\end{equation*}
is the boundary weight map of a unital CP-flow $\alpha$ over $\C^n$
whose minimal flow dilation $\alpha^d$ is an $E_0$-semigroup of type
II$_0$. Furthermore, the generalized boundary representation $\pi_t$ for $\alpha$ satisfies
\begin{equation*}
\pi_t(B) = \phi(I + \nu_t(\Lambda) \phi)^{-1}(\Omega_{\nu_t}(B))
\end{equation*}
for all $t>0$ and $B \in B(H)$.
\end{prop}

In the above proposition, we used the canonical identification
$B(H) \simeq M_n(B(L^2(0, \infty)))$.  Under this identification, the map $\Lambda(I_{\C^n})$
(i.e. $I_{B(\C^n)} \otimes \Lambda(I_\C)$) given by multiplication by $e^{-x}$ in $\cc^n \otimes L^2(0,\infty)$ can
be simply denoted by the diagonal matrix in $M_n(B(L^2(0,\infty)))$ whose $ii$ entry is $\Lambda=\Lambda(I_\cc)$
for each $i=1,\ldots, n$.  Thus one sees that $\mathfrak{A}(\cc^n \otimes L^2(0,\infty))$ is also
canonically isomorphic to $M_n(\mathfrak{A}(L^2(0,\infty)))$.
We note that in tensor notation, the map $\Omega_\nu$ defined in Proposition \ref{bdryweight}
is the map $I_{B(\C^n)} \otimes \nu$ from $M_n(\C) \otimes \fa(L^2(0, \infty))= \fa(\C^n \otimes L^2(0, \infty))$ to $M_n(\C)$.

\begin{defn}
A \emph{boundary weight double} is a pair $(\phi, \nu)$ where $\phi:M_n(\cc) \to M_n(\cc)$ is a unital $q$-positive map and $\nu$ is a Powers weight.  In the notation of the previous proposition, we call $\alpha^d$ the $E_0$-semigroup induced by the boundary weight double $(\phi, \nu)$.
\end{defn}

Motivated by the results
and terminology of \cite{powers-CPflows} and \cite{powers-holyoke}, we
define corners, $q$-corners, and hyper-maximal $q$-corners in an analogous context (Definitions
3.4 and 4.4 of \cite{jankowski1}):
\begin{defn}  Suppose $\phi: B(K_1) \rightarrow B(K_1)$ and $\psi: B(K_2) \rightarrow
B(K_2)$ are normal completely positive maps.  Write each $A \in B(K_1 \oplus K_2)
$ as $A=(A_{ij})$, where $A_{ij} \in B(K_j, K_i)$ for each
$i,j=1,2$.  We say a linear map $\gamma:
B(K_2, K_1) \rightarrow B(K_2, K_1)$ is a \emph{corner} from $\alpha$ to
$\beta$ if $\Theta: B(K_1 \oplus K_2) \rightarrow B(K_1 \oplus K_2)$
defined by
\begin{displaymath} \Theta \left( \begin{array}{cc} A_{11} & A_{12} \\
A_{21} & A_{22} \end{array}  \right) = \left(
\begin{array}{cc} \phi(A_{11}) & \gamma(A_{12}) \\ \gamma^*(A_{21})
& \psi(A_{22}) \end{array} \right)
\end{displaymath} is normal and completely positive.

We say that
$\gamma$ is
a \emph{$q$-corner} from $\phi$ to $\psi$ if $\Theta \geq_q 0$.  A
$q$-corner $\gamma$ is \emph{hyper-maximal} if, whenever
\begin{displaymath}
\Theta \geq_q  \left( \begin{array}{cc} \phi' & \gamma
\\ \gamma^* & \psi' \end{array} \right) \geq_q 0, \end{displaymath}
we have $\phi=\phi'$ and $\psi=\psi'$.
\end{defn}

Hyper-maximal $q$-corners between unital $q$-positive maps $\phi$ and $\psi$
allow us to compare the $E_0$-semigroups induced
by $(\phi, \nu)$ and $(\psi, \nu)$ if $\nu$ is a particular kind
of type II Powers weight (Proposition 4.6 of \cite{jankowski1}).

\begin{prop}\label{hypqc}
Let $\phi: M_n(\C) \rightarrow M_n(\C)$ and $\psi: M_k(\C)
\rightarrow M_k(\C)$ be unital $q$-positive maps, and let $\nu$ be a
type II Powers weight of the form \begin{equation*}\nu \Big((I - \Lambda)^\frac{1}{2}
B (I - \Lambda)^\frac{1}{2}\Big)=(f, Bf) \end{equation*}
where $f \in L^2(0,\infty)$ is a unit vector. The boundary weight
doubles $(\phi, \nu)$ and $(\psi, \nu)$ induce cocycle conjugate
$E_0$-semigroups if and only if there is a hyper-maximal $q$-corner
from $\phi$ to $\psi$.
\end{prop}

If $\phi: M_n(\C) \rightarrow M_n(\C)$ is a unital $q$-positive map
and $U \in M_n(\C)$ is any unitary matrix, then the map $\phi_U(A) := U^*
\phi(UAU^*)U$ is also unital and $q$-positive (Proposition 4.5 of \cite{jankowski1}).
We have the following definition from \cite{jankowski3}.
\begin{defn}
Let $\phi, \psi: M_n(\C) \rightarrow M_n(\C)$ be $q$-positive maps.  We say
$\phi$ is \emph{conjugate} to $\psi$ if $\psi=\phi_U$ for some unitary $U \in M_n(\C)$.
\end{defn}

 If $\phi: M_n(\C) \rightarrow M_n(\C)$
is unital and $q$-positive, then the map $\gamma: M_n(\C) \rightarrow
M_n(\C)$ defined by $\gamma(A)=\phi(AU^*)U$ is a hyper-maximal
$q$-corner from $\phi$ to $\phi_U$ (for details, see the discussion before
Proposition 2.11 of
\cite{jankowski3}).  Applying Proposition \ref{hypqc} gives us:

\begin{prop}\label{arrgh} Let $\phi: M_n(\C) \rightarrow M_n(\C)$
be unital and $q$-positive, and suppose $\psi$
is conjugate to $\phi$.  If $\nu$ is a type II Powers weight
of the form \begin{equation*}\nu\Big((I - \Lambda)^\frac{1}{2}B
(I - \Lambda)^\frac{1}{2}\Big)=(f,Bf),\end{equation*} then $(\phi, \nu)$ and $(\psi, \nu)$ induce
cocycle conjugate $E_0$-semigroups.
\end{prop}

We will generalize Proposition \ref{arrgh} substantially in Section 3, finding that if $\nu$ is an arbitrary type II Powers
weight and $\phi$ and $\psi$ are conjugate unital $q$-positive maps, then $(\phi, \nu)$
and $(\psi, \nu)$ induce \emph{conjugate} $E_0$-semigroups (Theorem~\ref{conjugation-thm} and Corollary~\ref{uniflow}).

Several cocycle conjugacy results for $E_0$-semigroups have been obtained through
the use of Proposition \ref{hypqc}.  For example, we have the following
(see Proposition~3.3 and Theorem 3.8 of \cite{jankowski2}).

\begin{thm}\label{statesbig}
A unital rank one linear map $\phi: M_n(\C) \rightarrow M_n(\C)$ is $q$-positive if and only if $\phi(A)=\rho(A) I$ where $\rho$ is a state, and $\phi$ is $q$-pure if and only if in addition $\rho$ is faithful.

Let $\phi: M_n(\C) \rightarrow M_n(\C)$ and
$\psi: M_{n'}(\C) \rightarrow M_{n'}(\C)$ be rank one unital $q$-positive
maps, and let $\nu$
be a type II Powers weight of the form \begin{equation*}\nu
\Big((I - \Lambda)^\frac{1}{2}B
(I - \Lambda)^\frac{1}{2}\Big) = (f,Bf).
\end{equation*}
Then the $E_0$-semigroups induced by
$(\phi, \nu)$ and $(\psi, \nu)$ are cocycle conjugate if and only if $\phi$
and $\psi$ are conjugate. 

Furthermore, if the support projection $P$ of $\rho$ satisfies $\rank P>1$ and $\mu$ is any Powers weight, 
then the E$_0$-semigroups induced by $(\phi, \nu)$ and $\mu$ are not cocycle conjugate.
\end{thm}

Theorem \ref{statesbig} shows that rank one $q$-positive maps are extremely
fruitful in constructing non-cocycle conjugate $E_0$-semigroups using boundary
weight doubles.  In section 3, we will extend Theorem \ref{statesbig} to a conjugacy result
(Theorem \ref{rankoneconjugacy}).

In \cite{jankowski1}, a necessary and sufficient condition was found
for a unital invertible map to be $q$-positive (see section 2.2 and
Proposition~6.1 of \cite{jankowski1}), and the invertible unital  $q$-pure
maps were entirely classified up to conjugacy (Theorem~6.11 of \cite{jankowski1}).
In contrast to the rank one case,
boundary weight doubles that combine
unital invertible $q$-pure maps with type II Powers weights of the form $\pure$
all induce cocycle conjugate $E_0$-semigroups:

\begin{thm}\label{inverts}  An invertible unital linear map $\psi: M_n(\C) \rightarrow M_n(\C)$
is $q$-pure
if and only if it is
conjugate to a Schur map $\phi$ that satisfies

\begin{equation*} \phi(a_{jk}e_{jk}) = \left\{
\begin{array}{cc}
\dfrac{a_{jk}}{1+i(\lambda_j - \lambda_k)}e_{jk} & \textrm{if } j<k \\ \noalign{\medskip}
a_{jk}e_{jk} & \textrm{if } j=k \\  \noalign{\medskip}
\dfrac{a_{jk}}{1-i(\lambda_j - \lambda_k)}e_{jk}& \textrm{if } j>k
\end{array} \right.
\end{equation*}
for all $j,k=1,\ldots, n$ and all $A=\sum a_{ij}e_{ij} \in M_n(\C)$,
where $\lambda_1, \ldots, \lambda_n \in \R$ and $\sum_{j=1}^n
\lambda_j = 0$.

If $\nu$ is a type II Powers weight of the form
$\nu((I - \Lambda)^\frac{1}{2} B(I - \Lambda)^\frac{1}{2})=(f,Bf)$, then the $E_0$-semigroup
induced by $(\psi,\nu)$ is cocycle conjugate to the $E_0$-semigroup induced by $(\imath_\C, \nu)$
for $\imath_\C$ the identity map on $\C$.
\end{thm}

We will show that the conclusion of this theorem holds if $\nu$ is an arbitrary type II Powers weight (Theorem~\ref{invertiblecocycleconjugacy}).

\section{Generalized Schur maps} \label{generalized-Schur}

Recall that a map $\phi: M_n(\cc) \to M_n(\cc)$ is said to be a \emph{Schur map} if there exists a matrix $Q=(q_{ij}) \in M_n(\cc)$ such that
$$
\phi \big( (x_{ij}) \big) = ( q_{ij} x_{ij} )
$$
In this section we consider a slight generalization of Schur maps and their relationship to  corners between CP-flows and other maps, encapsulating
results that are frequently needed in the comparison theory of CP-flows and in the remainder of the paper.

Powers has defined a related concept of Schur diagonal boundary weight maps (see Definition 4.31 of \cite{powers-CPflows}),
which is a special case of generalized Schur maps as defined below. A boundary weight map
$\omega:B(\cc^n)_* \to \fa(\cc^n \otimes L^2(0,\infty))_*$  is Schur diagonal in the sense
of Powers if and only if it is a generalized Schur map with respect to the decompositions
$\cc \oplus \cc \oplus \cdots \oplus \cc$ and $L^2(0,\infty) \oplus L^2(0,\infty) \oplus \cdots \oplus L^2(0,\infty)$ according to the following definition.

For each $i=1, 2, \dots, n$, let $K_i$ and $H_i$ be Hilbert spaces, and let $K=\bigoplus_{i=1}^n K_i$ and $H=\bigoplus_{i=1}^n H_i$. Let for $i=1,\dots, n$, $V_i:K_i \to K$ and $W_i:H_i \to H$ be the canonical isometries. Given operators $A \in B(K)$ and $B \in B(H)$, and  for $i,j=1,2, \dots, n$ given operators  $X \in B(K_j, K_i), Z \in B(H_j, H_i)$, we define
\begin{align*}
A_{ij} & = V_i^*AV_j \in B(K_j, K_i) &  X^{ij} & = V_iXV_j^* \in B(K) \\
B_{ij} & = W_i^*BW_j \in B(H_j, H_i) & Z^{ij} & = W_iZW_j^* \in B(H)
\end{align*}
In particular,
$$
(X^{ij})_{rs} = \delta_{ir} \delta_{js} X.
$$

Given a subalgebra $\fa$ of $B(H)$, and for each $i,j=1,2, \dots, n$, let $\fa_{ij} = W_i^*\fa W_j$. Suppose that for all $i,j=1,2, \dots, n$,
\begin{equation}\label{contains-corners}
W_i\fa_{ij}W_j^* \subseteq \fa.
\end{equation}
Given a linear map $\phi: \fa \to B(K)$, for each $i,j=1,2, \dots, n$ we define the linear map
$\phi_{ij} : \fa _{ij} \to B(K_j, K_i)$ given by
$$
\phi_{ij}(X) = [\phi(X^{ij})]_{ij}
$$
We say that $\phi$ is a \emph{generalized Schur map} with respect to the decompositions $\bigoplus_{i=1}^n K_i$ and $\bigoplus_{i=1}^n H_i$ if for all $A \in \fa$,
$$
[\phi (A)]_{ij} = \phi_{ij}( A_{ij}) .
$$
In particular, if $\phi$ is a generalized Schur map and if $X \in B(K_j, K_i)$, then
$$
\phi(X^{ij}) = [\phi_{ij}(X)]^{ij}.
$$

A similar definition applies to maps from $B(K)_*$ to the algebraic dual $\fa'$. If $\rho \in B(K)_*$ and $\eta \in \fa'$, we define for each $i,j=1,2, \dots, n$ the linear functionals $\rho_{ij} \in B(K_j, K_i)'$ and $\eta_{ij} \in \fa_{ij}'$ given by
$$
\rho_{ij} (X) = \rho(X^{ij}), \qquad \eta_{ij}(Z) = \eta(Z^{ij}),
$$
for all $X \in B(K_j, K_i)$ and $Z \in \fa_{ij}$. For each $\mu \in B(K_j, K_i)'$, we define $\mu^{ij} \in B(K)'$ given by
$$
\mu^{ij}(A) = \mu(A_{ij}).
$$
Given a map $\Psi: B(K)_* \to \fa'$ and $i,j=1,2, \dots, n$, we define $\Psi_{ij} : B(K_j, K_i)' \to \fa_{ij}'$ by
$$
\Psi_{ij}(\mu) = [\Psi(\mu^{ij})]_{ij}.
$$
We say that $\Psi : B(K)_* \to \fa'$ is a \emph{generalized Schur map with respect to the decompositions} $\bigoplus_{i=1}^n K_i$ and  $\bigoplus_{i=1}^n H_i$ if
$$
[\Psi(\rho)]_{ij} = \Psi_{ij}(\rho_{ij}).
$$
We observe that if $\Psi$ is a generalized Schur map and $\rho \in B(K)_*$, then
$\Psi([\rho_{ij}]^{ij})=[\Psi_{ij}(\rho_{ij})]^{ij}$.

If $\phi: B(H) \to B(K)$ is a normal generalized Schur map, then it follows that  $\widehat{\phi}$ is also a generalized Schur map: given $i,j=1,2, \dots, n$, $X \in B(K_j, K_i)$ and $\rho \in B(K)_*$,
\begin{align*}
[\widehat{\phi}(\rho)]_{ij} (X) & = \widehat{\phi}(\rho)(X^{ij}) = \rho( \phi(X^{ij})) = \rho([\phi_{ij}(X)]^{ij}) = \rho_{ij} (\phi_{ij}(X))  \\
&= \rho_{ij} ([\phi(X^{ij})]_{ij})
=  [\rho_{ij}]^{ij} (\phi(X^{ij}))
= \widehat{\phi}([\rho_{ij}]^{ij}) (X^{ij})
= [\widehat{\phi}([\rho_{ij}]^{ij})]_{ij}(X)\\
& = [\widehat{\phi}]_{ij}(\rho_{ij})(X)
\end{align*}
Analogously, if $\phi:B(H) \to B(K)$ is a normal map such that $\widehat{\phi}$ is a generalized Schur map,
then $\phi$ itself is also a generalized Schur map.

The following statement will be employed in the subsequent proposition. It refers to an elementary fact about completely positive maps on C$^*$-algebras, and we include a proof for convenience.

\begin{lem} \label{matrixzeroes}
Let $A, B$ be unital C$^*$-algebras and let $\phi: A \to B$ be a contractive completely positive map. Suppose that $\{e_i:i=1, \dots, n\}$ and $\{f_j : j=1, \dots, n\}$ are families of mutually orthogonal projections summing up to the identity in $A$ and $B$, respectively. If for all $i\neq j$,
$$
f_i\phi(e_j)f_i = 0
$$
then for all $i,j$, and $a \in A$,
$$
f_i\phi(e_iae_j)f_j= \phi(e_iae_j) = f_i \phi(a) f_j.
$$
\end{lem}
\begin{proof}
Observe that for $i\neq j$,
$$
f_i\phi(e_j)^*\phi(e_j)f_i \leq f_i\phi(e_j^*e_j)f_i = 0
$$
since $\phi$ is a completely positive contraction. Hence $\phi(e_j)f_i=0$, and by taking adjoints, $f_i\phi(e_j)=0$ for all $i\neq j$. Since $\sum_{j=1}^nf_j=1$, we have that
$$
\phi(e_j) = f_j\phi(e_j) f_j \leq f_j
$$
Let $a$ be a positive contraction in $A$. We have that for all $i,j$, and $a \in A$ a contraction,
\begin{align*}
[\phi(e_iae_j)]^*[\phi(e_iae_j)]
& \leq \phi(e_ja^*e_i) \phi(e_iae_j) \\
& \leq \phi(e_ja^*e_i e_i ae_j)
 \leq \phi(e_j)  \leq f_j
\end{align*}
Now observe that if $c\in A$ and $c^*c \leq f_j$, then $(1-f_j)c^*c(1-f_j)=0$ hence $cf_j=c$. Thus we have that $\phi(e_iae_j)=\phi(e_iae_j)f_j$. Thus we conclude by taking adjoints and reapplying the identity that for all $i,j$, for all contractions $a \in A$,
$$
f_i\phi(e_iae_j)f_j = \phi(e_iae_j).
$$
Thus
$$
f_i\phi(a)f_j = f_i \left( \sum_{r,s=1}^n \phi(e_r a e_s) \right) f_j =
f_i \left(\sum_{r,s=1}^n  f_r\phi(e_r a e_s)f_s \right) f_j = f_i\phi(e_iae_j)f_j.
$$
The lemma now follows by linearity.
\end{proof}

\begin{prop} \label{gen-schur-flow-weight}
For each $i=1, \dots, n$, let $K_i$ be a separable Hilbert space, and let $H_i = K_i \otimes L^2(0,\infty)$. Define
$K=\bigoplus_{i=1}^n K_i$ and $H=\bigoplus_{i=1}^n H_i = K \otimes L^2(0,\infty)$. Suppose that $\alpha$ is a CP flow over $K$ with  boundary weight map $\omega: B(K)_* \to  \fa(H)_*$ and generalized boundary representation $\pi_t$ for $t>0$.
Then $\alpha_t$ is a generalized Schur map with respect to $\bigoplus_{i=1}^n H_i$ for every $t>0$ if and only if
$\omega$ as well as  $\omega_t$ and $\pi_t$ are generalized Schur maps with respect to $\bigoplus_{i=1}^n K_i$ and
$\bigoplus_{i=1}^n H_i$ for every $t>0$.
\end{prop}
\begin{proof}
We will make use of the notation introduced in the definition of generalized Schur maps. 
Denote by $S^{(i)}_t$ the right shift semigroup on $H_i$, for $i=1,2$ and let $S_t=S^{(1)}_t \oplus S^{(2)}_t$ be the right shift semigroup on $H$. We remark that the null boundary algebra
$\fa(H)$ satisfies the property \eqref{contains-corners}, because $\Lambda(I_K) = \bigoplus_{i=1}^n \Lambda(I_{K_i})$. Thus it makes sense to claim that $\omega$ is a generalized Schur map with respect to $\bigoplus_{i=1}^n K_i$ and $\bigoplus_{i=1}^n H_i$.

Let $\Gamma$ be the map defined by \eqref{gamma}.
It is easy to check that $\Gamma$ is a generalized Schur map on $B(H)$ with respect to $\bigoplus_{i=1}^n H_i$,
hence $\hatG$ is a generalized Schur maps on $B(H)_*$ with respect to  $\bigoplus_{i=1}^n H_i$.  Similarly, $\Lambda: B(K) \to B(H)$  and $\widehat{\Lambda}$ are generalized Schur maps with respect to the decompositions $\bigoplus_{i=1}^n K_i$ and $\bigoplus_{i=1}^n H_i$.

Suppose that $\alpha_t$ is a generalized Schur map for every $t>0$, and let $R_\alpha$ be its resolvent defined by \eqref{resolvent-pure}.  It is clear that $R_\alpha$ and $\widehat{R}_{\alpha}$ are generalized Schur maps
with respect to  $\bigoplus_{i=1}^n H_i$.
Let $\rho \in B(K)_*$ be fixed, and let $\eta \in B(H)_*$ be any normal functional such that
$\rho=\widehat{\Lambda}(\eta)$, and therefore $\rho_{ij}=\widehat{\Lambda}_{ij}(\eta_{ij})$. Then by \eqref{got-omega}, we have that for all $x>0$ and $T \in S^{(i)}_x {S^{(i)}_x}^* B(H) S^{(j)}_x {S^{(j)}_x}^*$, and for all $i,j=1,2, \dots,n$,
\begin{align} \nonumber
\left[\omega(\rho)\right]_{ij}(T) & = \omega(\rho)(T^{ij}) =
\lim_{y \to x+} \frac{1}{y-x}
(\widehat{R}_\alpha - \widehat{\Gamma})(\eta) ( T^{ij} - e^{x-y} S_{y-x}T^{ij}S_{y-x}^*) \\ \nonumber
& = \lim_{y \to x+} \frac{1}{y-x}
(\widehat{R}_\alpha - \widehat{\Gamma})_{ij}(\eta_{ij}) ( T - e^{x-y} S^{(i)}_{y-x}TS^{(j)^*}_{y-x}) \\ \nonumber
& = \lim_{y \to x+} \frac{1}{y-x}
\left[(\widehat{R}_\alpha - \widehat{\Gamma})([\eta_{ij}]^{ij})\right]_{ij} ( T - e^{x-y} S^{(i)}_{y-x}TS^{(j)^*}_{y-x}) \\ \nonumber
& = \lim_{y \to x+} \frac{1}{y-x}
(\widehat{R}_\alpha - \widehat{\Gamma})([\eta_{ij}]^{ij}) (T^{ij} - e^{x-y} S_{y-x}T^{ij}S_{y-x}^*) \\
  \label{gotten}
&= \omega([\rho_{ij}]^{ij})(T^{ij})= \omega_{ij}(\rho_{ij})(T).
\end{align}

If $A \in \fa_{ij}$, then for the operators $\{A_x\}_{x>0}$ defined by 
$A_x= S^{(i)}_x {S^{(i)}_x}^* A S^{(j)}_x {S^{(j)}_x}^*$, we have $A_x^{ij}= S_x S_x^* A^{ij} S_x S_x^*$,
so using equation \eqref{gotten} and the fact that $\omega(\rho) \in \mathfrak{A}(H)_*$
since $\rho \in
B(K)_*$, we obtain by Proposition~\ref{get-omega} that
\begin{align*}
[\omega(\rho)]_{ij}(A) & = \omega(\rho)(A^{ij})= \lim_{x \rightarrow 0+} \omega(\rho)(A_x ^{ij})
= \lim_{x \rightarrow 0+} \omega([\rho_{ij}]^{ij})(A_x^{ij})= \omega([\rho_{ij}]^{ij})(A^{ij}) \\
& = \omega_{ij}(\rho_{ij})(A).
\end{align*}
Thus we have shown that $[\omega(\rho)]_{ij}= \omega_{ij}(\rho_{ij})$, hence $\omega$ is a generalized Schur map with respect to the decompositions $\bigoplus_{i=1}^n K_i$ and $\bigoplus_{i=1}^n H_i$. It follows immediately that for every $t>0$, $\omega_t$ is also a generalized Schur map with respect to the decompositions $\bigoplus_{i=1}^n K_i$ and $\bigoplus_{i=1}^n H_i$.

For each $t>0$, by Theorem~\ref{powerstheorem}, $\hat{\pi}_t= \omega_t(I + \hat{\Lambda}\omega_t)^{-1}$. A simple computation shows that the inverse of a generalized Schur map is also a generalized Schur map (with respect the reverse decompositions), whereby it follows from the previous paragraph that $\widehat{\pi}_t$ is a generalized Schur map with respect to the decompositions $\bigoplus_{i=1}^n K_i$ and $\bigoplus_{i=1}^n H_i$.
Therefore, by the observation preceding Lemma~\ref{matrixzeroes}, $\pi_t$ is a generalized Schur map for every $t>0$.

Conversely, suppose that $\omega$ is a generalized Schur map.  It follows trivially that $\omega_t$ is a generalized Schur map
for every $t>0$, and the argument given in the previous paragraph shows that
$\pi_t$ is a generalized Schur map for every $t>0$.  By equation \eqref{resolvent}, $\widehat{R}_\alpha$
is the composition of generalized Schur maps with respect
to the decomposition $\bigoplus_{i=1}^n H_i$
and is thus a generalized Schur map,
hence $R_\alpha$ is a generalized Schur map with respect to the decomposition $\bigoplus_{i=1}^n H_i$.
For each $i=1, \ldots, n$, let $W_i:H_i \to H$ be the canonical isometric embedding and let $E_i =W_iW_i^*\in B(H)$ be the projection onto the subspace of $H$ associated to $H_i$.
If $i, j \in \{1, \ldots, n\}$ and $i \neq j$,  then $(E_i)_{jj}=W_j^*E_iW_j=W_j^*E_jE_iE_jW_j=0$, and from equation \eqref{resolvent} and the fact that $R_\alpha$
is a generalized Schur map, it follows that
$$
0= (R_\alpha)_{jj}\Big((E_i)_{jj}\Big) = \Big[R_\alpha(E_i)\Big]_{jj} = \Big[ \int_0 ^\infty e^{-t} \alpha_t(E_i) dt \Big]_{jj}
= \int_0 ^\infty e^{-t} [\alpha_t(E_i)]_{jj} dt.
$$
Since $[\alpha_t(E_i)]_{jj}=W_j^*\alpha_t(E_i)W_j$ is a positive operator in $B(H_j)$ for every $t \geq 0$,
the above equation implies that $[\alpha_t(E_i)]_{jj}=0$ for every $t \geq 0$.  Now let $t \geq 0$
be fixed. Note that $E_i A E_j
= (A_{ij})^{ij}$ for all $A \in B(H)$, so
\begin{equation}\label{eis} E_j \alpha_t(E_i)E_j = \Big([\alpha_t(E_i)]_{jj}\Big)^{jj} = 0 \end{equation}
whenever $i \neq j$.  Note that the projections $E_1, \ldots, E_n$ are mutually orthogonal
and sum to $I$.  Therefore, by Lemma \ref{matrixzeroes} and equation \eqref{eis} we have
$E_i \alpha_t(A) E_j = E_i \alpha_t(E_i A E_j) E_j$
for all $A \in B(H)$ and $i,j=1, \ldots n$, so
\begin{align*}
\Big([\alpha_t(A)]_{ij}\Big)^{ij} = E_i \alpha_t(A) E_j = E_i \alpha_t(E_i A E_j) E_j  = E_i \Big( \alpha_t[(A_{ij})^{ij}] \Big)
E_j = \Big[(\alpha_t)_{ij}(A_{ij})\Big]^{ij},
\end{align*}
hence $[\alpha_t(A)]_{ij} = (\alpha_t)_{ij}(A_{ij})$.
\end{proof}

The following proposition will be useful in analyzing flow corners:

\begin{prop}\label{corner-computing}
For each $i=1, \dots, n$, let $K_i$ be a separable Hilbert space, and let $H_i = K_i \otimes L^2(0,\infty)$.
Define $K=\bigoplus_{i=1}^n K_i$ and $H=\bigoplus_{i=1}^n H_i = K \otimes L^2(0,\infty)$.
Let $\vartheta$ and $\vartheta'$ be CP flows over $K$ with boundary weight maps $\omega$ and $\omega'$,
and generalized boundary representations $\pi_t$ and $\pi_t'$. Suppose that $\vartheta_t$
and $\vartheta_t'$ are generalized Schur maps for every $t>0$, and let $i,j \in \{1, \ldots, n\}$.
Then $[\vartheta_t]_{ij}=[\vartheta'_t]_{ij}$  for all $t>0$
if and only if, for every $\rho \in B(K)_*$, $A \in B(H)$, and $t>0$,
$$
[\omega(\rho)]_{ij} = [\omega'(\rho)]_{ij}, \qquad  [\pi_t(A)]_{ij} =  [\pi'_t(A)]_{ij}.
$$
\end{prop}
\begin{proof}
Suppose that for some $i$ and $j$ we have $[\vartheta_t]_{ij}=[\vartheta'_t]_{ij}$ for all
$t \geq 0$.  It clearly follows that $[R_\vartheta]_{ij} = [R_{\vartheta'}]_{ij}$.
Let $\rho \in B(K)_*$ be arbitrary and let $\eta \in B(H)_*$ be such that $\rho = \widehat{\Lambda}\eta$. It follows from equation \eqref{gotten} that for all $x>0$ and $T \in S^{(i)}_x {S^{(i)}_x}^* B(H) S^{(j)}_x {S^{(j)}_x}^*$,
\begin{align*}
\left[\omega(\rho)\right]_{ij}(T)
& = \lim_{y \to x+} \frac{1}{y-x}
(\widehat{R}_\vartheta - \widehat{\Gamma})([\eta_{ij}]^{ij}) (T^{ij} - e^{x-y} S_{y-x}T^{ij}S_{y-x}^*) \\
& =\lim_{y \to x+} \frac{1}{y-x}\;
\eta_{ij} \Big( (R_{\vartheta} - \Gamma)_{ij} (T^{ij} - e^{x-y} S_{y-x}T^{ij}S_{y-x}^*)_{ij} \Big) \\
& =\lim_{y \to x+} \frac{1}{y-x}\;
\eta_{ij} \Big( (R_{\vartheta'} - \Gamma)_{ij} (T^{ij} - e^{x-y} S_{y-x}T^{ij}S_{y-x}^*)_{ij} \Big) \\
& =\lim_{y \to x+} \frac{1}{y-x}\;
(\widehat{R}_{\vartheta'} - \widehat{\Gamma})([\eta_{ij}]^{ij}) (T^{ij} - e^{x-y} S_{y-x}T^{ij}S_{y-x}^*) \\
& = \left[\omega'(\rho)\right]_{ij}(T),
\end{align*}
hence $\left[\omega(\rho)\right]_{ij}=\left[\omega'(\rho)\right]_{ij}$ by Proposition~\ref{get-omega}.

It follows by a simple computation that  for all $t>0$, $\left[\omega_t(\rho)\right]_{ij}=\left[\omega_t'(\rho)\right]_{ij}$. Now observe that for all $t>0$, $\widehat{\pi}_t =  \omega_t(I + \hat{\Lambda}\omega_t)^{-1}$, hence for all $\rho \in B(K)_*$, $[\widehat{\pi}_t(\rho)]_{ij}=[\widehat{\pi}'_t(\rho)]_{ij}$. Thus for all $t>0$, $\rho \in B(K)_*$ and $A \in B(H)$,
\begin{align*}
\rho\Big( ([\pi_t(A)]_{ij})^{ij} \Big) & = (\rho_{ij})^{ij}(\pi_t(A)) = \widehat{\pi}_t((\rho_{ij})^{ij})(A) \\
& = \widehat{\pi}'_t((\rho_{ij})^{ij})(A) = \rho\Big( ([\pi'_t(A)]_{ij})^{ij}\Big).
\end{align*}
Hence  $([\pi_t(A)]_{ij})^{ij}= ([\pi'_t(A)]_{ij})^{ij}$, from which the desired identity follows.

Conversely, suppose that for some $i,j \in \{1, \ldots, n\}$, we have $[\omega(\rho)]_{ij} = [\omega'(\rho)]_{ij}$
for all $\rho \in B(K)_*$.  The argument from the previous paragraph shows that $[\pi_t(A)]_{ij} =  [\pi'_t(A)]_{ij}$
for all $A \in B(H)$ and $t >0$.   Since $\vartheta_t$ and $\vartheta_t'$ are generalized Schur maps,
so are $R_\vartheta$ and $R_{\vartheta'}$.  Observe that for all $\eta \in B(H)_*$,
\begin{align*}
[\hat{R}_\vartheta (\eta)]_{ij}
& = \hat{\Gamma}_{ij}([\omega(\hat{\Lambda}\eta)]_{ij} + \eta_{ij})
 =  \hat{\Gamma}_{ij}([\omega'(\hat{\Lambda}\eta)]_{ij} + \eta_{ij})
= [\hat{R}_{\vartheta'} (\eta)]_{ij}.
\end{align*}
Therefore, it follows that $[\widehat{R}_\vartheta]_{ij} = [\widehat{R}_{\vartheta'}]_{ij}$.

Now define the continuous idempotent $L: B(H)_* \to B(H)_*$ given by
$$
L(\rho) = (\rho_{ij})^{ij}
$$
and let $\fm$ be the range of $L$. Observe that for every $t>0$, $\widehat{\vartheta}_t$, $\widehat{\vartheta}'_t$ as well as $\widehat{R}_\vartheta$ and $\widehat{R}_{\vartheta'}$ are generalized Schur maps, hence $\fm$ is a closed invariant subspace of $B(H)_*$ for those maps. Furthermore, on $\fm$, the restriction of $\widehat{\vartheta}$ and $\widehat{\vartheta}'$ constitute C$_0$-semigroups whose resolvents are precisely the restriction of $\widehat{R}_\vartheta$ and $\widehat{R}_{\vartheta'}$ to $\fm$. In addition, note that for all $\rho \in B(H)_*$,
\begin{align*}
\widehat{R}_\vartheta(L(\rho)) & = \widehat{R}_\vartheta((\rho_{ij})^{ij})
= [\widehat{R}_\vartheta]_{ij}(\rho_{ij})
= [\widehat{R}_{\vartheta'}]_{ij}(\rho_{ij})
= \widehat{R}_{\vartheta'}((\rho_{ij})^{ij}) = \widehat{R}_{\vartheta'}(L(\rho)).
\end{align*}
Thus the resolvents of the C$_0$-semigroups $\widehat{\vartheta}|_\fm$ and $\widehat{\vartheta}'|_\fm$ coincide.
It follows that $\widehat{\vartheta}|_\fm=\widehat{\vartheta}'|_\fm$. Thus for every $t>0$, $X \in B(H)$, and $\rho \in B(H)_*$,
\begin{align*}
\widehat{\vartheta}_t(L(\rho))(X) & = L(\rho)(\vartheta_t(X))
= [\rho_{ij}]^{ij}(\vartheta_t(X)) = \rho_{ij}([\vartheta_t(X)]_{ij})
= \rho\left(([\vartheta_t(X)]_{ij})^{ij}\right).
\end{align*}
We have the analogous identity for $\vartheta'$, hence we have that for all $X\in B(H)$ and $\rho \in B(H)_*$,
$$
\rho\left(([\vartheta_t(X)]_{ij})^{ij}\right) = \rho\left(([\vartheta'_t(X)]_{ij})^{ij}\right).
$$
It follows that $[\vartheta_t(X)]_{ij}=[\vartheta'_t(X)]_{ij}$ for all $X\in B(H)$. Since both $\vartheta_t$ and $\vartheta'_t$ are generalized Schur maps, we obtain that $[\vartheta_t]_{ij} = [\vartheta'_t]_{ij}$.
\end{proof}

\begin{remark} \label{summary} Propositions \ref{gen-schur-flow-weight}
and \ref{corner-computing} are used very frequently in the remainder of the article.
For example, if $\omega: B(K_1 \oplus K_2)_* \to \fa(H_1 \oplus H_2)_*$ is a $q$-positive boundary weight map
which is a generalized Schur map with respect to the decompositions $K_1 \oplus K_2$ and $H_1 \oplus H_2$,
then Proposition \ref{gen-schur-flow-weight} implies that $\omega$ induces a CP-flow $\Theta$ over $K_1
\oplus K_2$ consisting of generalized Schur maps and that the generalized boundary
representation $\Pi_t$ for $\Theta$ consists of generalized Schur maps.  Therefore, $\Theta$ has the form
\begin{equation} \label{moretheta}
\Theta_t \begin{pmatrix} A & B \\ C & D \end{pmatrix} = \begin{pmatrix} \alpha_t(A) & \sigma_t(B)
\\ \sigma^*_t(C) &  \beta_t(D) \end{pmatrix},
\end{equation}
where $\alpha_t$ maps $B(H_1)$ into itself, $\sigma_t$ maps $B(H_2, H_1)$ into itself, and $\beta_t$ maps $B(H_2)$ into itself for all $t \geq 0$.  Note that since $\Theta$ is a CP-flow, the semigroups
$\alpha=\{\alpha_t\}_{t \geq 0}$ and
$\beta= \{\beta_t\}_{t \geq 0}$ must be CP-flows over $K_1$ and $K_2$, respectively, hence
$\sigma$ is a flow corner from $\alpha$ to $\beta$.  Let $\mu$ be the boundary
weight map and $\pi_t$ be the generalized boundary representation for $\alpha$, and let
$\eta$ be the boundary weight map and $\xi_t$ be the generalized boundary representation for $\beta$.  Since $(R_\Theta)_{11}= R_\alpha$
and $(R_\Theta)_{22} = R_\beta$, it follows from \eqref{gotten} that for all $\rho \in B(K_1 \oplus K_2)_*$,
$$[\omega(\rho)]_{11} = \mu(\rho_{11}), \qquad [\omega(\rho)]_{22} = \eta(\rho_{22}).$$
By the above line and the fact that $\Pi_t$ is a generalized Schur map for every $t>0$,
it follows that $\Pi_t$ has the form
\begin{equation} \label{morepi}
\Pi_t \begin{pmatrix} A & B \\ C & D \end{pmatrix} = \begin{pmatrix} \pi_t(A) & \gamma_t(B) \\ \gamma^*_t(C) &  \xi_t(D) \end{pmatrix}
\end{equation}
for some family $\gamma= \{\gamma_t\}_{t > 0}$ of contractions from $B(H_2, H_1)$ into $B(K_2, K_1)$.

Suppose that $\Theta \geq \Theta'$ for a CP-flow $\Theta'$ over $K_1 \oplus K_2$ of the form
\begin{equation*}
\Theta_t' \begin{pmatrix} A & B \\ C & D \end{pmatrix} = \begin{pmatrix} \alpha_t'(A) & \sigma_t(B)
\\ \sigma^*_t(C) &  \beta_t'(D) \end{pmatrix}.
\end{equation*}
Note that $\alpha'$ and $\beta'$ are CP-flows over $K_1$
and $K_2$ which are subordinate to $\alpha$ and $\beta$, respectively. By the previous paragraph, we have that if $\Pi'_t$ is the generalized boundary representation for $\Theta'$, then
\begin{equation*}
\Pi_t' \begin{pmatrix} A & B \\ C & D \end{pmatrix} =
\begin{pmatrix} \pi_t'(A) & \gamma'_t(B) \\ (\gamma')_t^*(C) &  \xi_t'(D) \end{pmatrix}.
\end{equation*}
Furthermore, since $[\Theta_t]_{12} = [\Theta'_t]_{12}$
for all $t \geq 0$, Proposition \ref{corner-computing} implies that $[\Pi_t]_{12}(B)=[\Pi_t(B^{12})]_{12} = [\Pi_t'(B^{12})]_{12}=[\Pi_t']_{12}(B)$
for all $B \in B(H_2, H_1)$, hence $\Pi_t'$ has the form
\begin{equation*}
\Pi_t' \begin{pmatrix} A & B \\ C & D \end{pmatrix} =
\begin{pmatrix} \pi_t'(A) & \gamma_t(B) \\ \gamma_t^*(C) &  \xi_t'(D) \end{pmatrix}
\end{equation*}
for all $t \geq 0$.

Conversely, if $\alpha$ and $\beta$ are CP-flows over $K_1$ and $K_2$, respectively,
and if $\sigma$ is a flow corner from $\alpha$ to $\beta$,
then equation \eqref{moretheta} defines a CP-flow $\Theta$ over $K_1 \oplus K_2$.  By Proposition \ref{gen-schur-flow-weight}, the boundary
weight map $\omega$ and generalized boundary representation $\Pi_t$ for $\Theta$ are generalized Schur maps.
Let $\mu$ be the boundary
weight map and $\pi_t$ be the generalized boundary representation for $\alpha$, and let
$\eta$ be the boundary weight map and $\xi_t$ be the generalized boundary representation for $\beta$.  The same argument
given in the first paragraph of this remark shows that for all $\rho \in B(K_1 \oplus K_2)_*$
we have $[\omega(\rho)]_{11} = \mu(\rho_{11})$ and
$[\omega(\rho)]_{22} = \eta(\rho_{22})$, and that in addition $\Pi_t$ has the form \eqref{morepi}.
\end{remark}

As a further application of generalized Schur maps, we clarify the relationship between flow corners and matrices of flow corners.

\begin{prop}\label{submatrix-flow-corners}
For each $i=1,2, \dots n$, let $K_i$ be a separable Hilbert space, and
 let $H_i = K_i \otimes L^2(0,\infty)$. Let $K=\bigoplus_{i=1}^n K_i$ and $H=\bigoplus_{i=1}^n H_i$. Suppose that $\Theta=(\theta^{(ij)})_{i,j=1}^n$ is a one-parameter semigroup from $B(H)$ to $B(H)$ such that for each $i,j=1,2,\dots, n$, $i\neq j$, and $t>0$,  the map given by
$$
\Psi^{(ij)}_t\begin{pmatrix} A & B \\ C & D \end{pmatrix} = \begin{pmatrix} \theta^{(ii)}_t(A) & \theta^{(ij)}_t(B) \\ \theta^{(ji)}_t(C) &  \theta^{(jj)}_t(D) \end{pmatrix}
$$
defines a unital CP-flow over $K_i \oplus K_j$.

$\Theta$ is a CP-flow over $K$ if only if the boundary weight map $\omega$ defined by
\begin{equation}\label{itisschur}
[\omega(\rho)]_{ij} = \omega_{ij}(\rho_{ij})
\end{equation}
for all $\rho \in B(K)_*$  and $i,j=1,2,\dots n$ is a $q$-positive boundary weight map, where the map given by
\begin{equation}\label{thebdmap}
\eta \mapsto \begin{pmatrix} \omega_{ii}(\eta_{11}) & \omega_{ij}(\eta_{12}) \\
\omega_{ji}(\eta_{21}) &  \omega_{jj}(\eta_{22}) \end{pmatrix}
\end{equation}
for all $\eta \in B(K_i \oplus K_j)_*$ is the boundary weight map for $\Psi^{(ij)}$. In this case, $\omega$ is the boundary weight map for $\Theta$.
\end{prop}
\begin{proof}
Suppose that $\Theta$ is a CP-flow over $K$. Then it has a $q$-positive boundary weight map $\omega$. It is clear that $\Theta$ is a generalized Schur map, hence $\omega$ is also a generalized Schur map, i.e. it has the form \eqref{itisschur}. It remains to show that for every $i,j=1,2, \dots, n$, $i\neq j$, the boundary weight map of $\Psi^{(ij)}_t$ is given by \eqref{thebdmap}.

Let us fix $i,j=1,2, \dots, n$,  $i\neq j$, let us temporarily denote $\Psi= \Psi^{(ij)}$ and let us define the canonical embedding $\varepsilon: B(H_i \oplus H_j) \to B(H)$ and the canonical compression $E:B(H) \to B(H_i \oplus H_j)$. Then it is clear that for all $t>0$,
$$
\Psi_t = E \circ \Theta_t \circ \varepsilon.
$$
Similarly, we have that
$$
R_\Psi = E \circ R_\Theta \circ \varepsilon, \quad \text{and} \quad \Gamma_{H_i\oplus H_j} = E \circ \Gamma \circ \varepsilon.
$$
Let $\omega'$ be the map given by \eqref{thebdmap}. Then by \eqref{resolvent} we have that for all $\rho \in B(H_i \oplus H_j)_*$,
\begin{align}\nonumber
\widehat{R}_\Psi(\rho) & = (\widehat{\varep} \circ \widehat{R_\Theta} \circ \widehat{E})(\rho)
= \widehat{\varep} \Big( \widehat{\Gamma}( \omega(\widehat{\Lambda}\widehat{E}\rho) + \widehat{E}(\rho) \Big) \\
\nonumber
& =  \widehat{\Gamma}_{H_i\oplus H_j}\Big(  \widehat{\varep}\omega(\widehat{\Lambda}\widehat{E}\rho) + \rho \Big) \\
\label{merry-go-round}
& =  \widehat{\Gamma}_{H_i\oplus H_j}\Big( \omega'(\widehat{\Lambda}_{H_i\oplus H_j}\rho) + \rho \Big).
\end{align}
Thus it follows from \eqref{resolvent} that $\omega'$ must be the boundary weight map of $\Psi$.

Conversely, suppose that equation \eqref{itisschur} defines a $q$-positive boundary weight $\omega$, where for each $i$ and $j$ with $i \neq j$, \eqref{thebdmap} is the boundary weight map for $\Psi^{(ij)}$. Then $\omega$ induces a CP-flow $\Theta'$ over $K$, and $\Theta_t'$ must be a generalized Schur map for every $t>0$. Now let $i,j=1,2,\dots, n$ be fixed, and let us denote $E$ and $\varep$ as in the previous paragraph. Then we have that $\Upsilon^{(ij)}_t=E \circ \Theta_t' \circ \varepsilon$ is a CP-flow over $K_i\oplus K_j$. This applies for every $i,j$.  Hence we can apply the forward part of the theorem to $\Theta'$ and $\Upsilon^{(ij)}$ for $i,j=1,2,\dots,n$, so that by \eqref{merry-go-round}, $\Upsilon^{(ij)}$ has boundary weight map given by \eqref{thebdmap}. It follows by the uniqueness of the boundary weight map (see \eqref{resolvent} and Proposition~\ref{get-omega}) that $\Upsilon^{(ij)}= \Psi^{(ij)}$. Thus we obtain that $\Theta = \Theta'$ and $\omega$ is its boundary weight map.
\end{proof}

\section{Cocycle conjugacy in the case of invertible $q$-pure maps}

In this section we will frequently make use of the canonical identifications
discussed after Proposition \ref{bdryweight}.
We remark that if $T \in M_n(\cc)$ is positive, and $\mu \in \mathfrak{A}(\cc^n
\otimes L^2(0,\infty))_*$ is a positive boundary weight, then the map $\mu_T(A) = \mu(T \otimes A)$
is a positive boundary weight in $\mathfrak{A}(L^2(0,\infty))_*$.  Indeed, $\mu_T$ is positive by
construction, and it is a boundary weight since
$$
\mu_T((I - \Lambda)^{1/2} B (I - \Lambda)^{1/2}) = \mu\Big( (I - I \otimes \Lambda)^{1/2} (T \otimes
B)(I - I \otimes \Lambda)^{1/2}\Big)
$$
for all $B \in B(L^2(0, \infty))$.  Under our matrix identifications, $\mu_T(A)=\mu(\sum_{i,j=1}^n t_{ij}A^{ij})$.

\begin{lem} \label{diags} Let $\phi: M_n(\C) \rightarrow M_n(\C)$ be a
unital $q$-positive Schur map.
Let $\nu$ be a type II Powers weight, and let
$\Theta$ be the unital CP-flow over $\C^n$ induced by the boundary weight
double $(\phi, \nu)$. Let $H=\C^n \otimes L^2(0, \infty)$,
identifying $B(H)$ with $M_n(B(L^2(0, \infty)))$.
Suppose $\Theta'$ is a CP-flow over $\C^n$ such that $\Theta \geq \Theta'$, and let
$\{\xi_t\}_{t>0}$ be the generalized boundary representation for $\Theta'$. Then
$\xi_t$ is a generalized Schur map with respect to the decompositions
$\bigoplus_{i=1}^n \C$ and $\bigoplus_{i=1}^n B(L^2(0, \infty))$ for every $t>0$.

Furthermore, for each
$k=1,2,\dots, n$, there exists a positive boundary weight
$\omega_{k} \in \mathfrak{A}(L^2(0, \infty))_*$ such that
\begin{equation} \label{pain} \xi_t(X^{kk})
= \frac{(\omega_k)_t(X)}{1 + (\omega_k)_t(\Lambda)} e_{kk}
\end{equation}
for all $X \in B(L^2(0, \infty))$,
where the matrices $\{e_{ij}\}_{i,j=1}^n$ are the standard
matrix units for $M_n(\C)$.
\end{lem}

\begin{proof}
Let $\pi=\{\pi_t\}_{t>0}$ be the boundary representation for $\Theta$,
and let $t>0$.  By Proposition \ref{bdryweight}, $\pi_t$ is given by
\begin{equation*}\label{pi}
\pi_t(B) = \phi(I + \nu_t(\Lambda) \phi)^{-1}(\Omega_{\nu_t}(B))
\end{equation*}
for all $B \in B(\cc^n \otimes L^2(0,\infty))$, and by
Theorem~\ref{boundary-representation-subordinates}, $\pi_t - \xi_t$ is completely positive.
Using the fact that $\phi$ is a unital
Schur map and  $\nu_t(I - \Lambda) \leq \nu(I-\Lambda) = 1$, we obtain
$$
\xi_t(I^{kk}) \leq \pi_t(I^{kk}) = \frac{\nu_t(I)}{1+ \nu_t(\Lambda)} e_{kk} \leq e_{kk}
$$
for every $k=1, \ldots, n$.  Therefore, $e_{jj} \xi_t(I^{kk}) e_{jj}=0$
if $j \neq k$. Fix $A \in B(H)$.
Note that $\xi_t$, $\{e_{jj}\}_{j=1}^n$, and $\{I^{kk}\}_{k=1}^n$ satisfy the conditions of
Lemma~\ref{matrixzeroes}, hence $e_{jj} \xi_t(A)e_{kk} = e_{jj} \xi_t(I^{jj} A I^{kk})e_{kk}$ for all
$j,k=1, \ldots, n$.  Therefore, for all $i,j=1, \ldots, n$, we have
$$\Big([\xi_t(A)]_{ij}\Big)^{ij} = e_{ii} \xi_t(A) e_{jj} = e_{ii} \xi_t(I^{ii} A I^{jj})e_{jj}
= e_{ii}\Big(\xi_t[(A_{ij})^{ij}]\Big) e_{jj}
= \Big[(\xi_t)_{ij}(A_{ij})]_{ij}\Big]^{ij},$$
hence $[\xi_t(A)]_{ij} = (\xi_t)_{ij}(A_{ij})$.  Therefore, $\xi_t$ is
a generalized Schur map with respect to the decompositions
$\bigoplus_{i=1}^n \C$ and $\bigoplus_{i=1}^n B(L^2(0, \infty))$ for every $t>0$.

Now let $k \in \{1, \ldots, n\}$ be arbitrary.  For simplicity of notation,
we will write $\omega$ rather than $\omega_k$ for
the boundary weight which, as we will show, satisfies equation \eqref{pain}.
Let $\rho_0 \in M_n(\C)^*$ be the state defined by $\rho_0(X)=x_{kk}$ for all $X=(x_{ij}) \in
M_n(\C)$.
Let $\eta$ be the boundary weight map for to $\Theta'$,
and for each $t>0$, let
$\Psi_t = (I - \xi_t \circ \Lambda)^{-1} \xi_t$, noting that
$\Psi_t$ is a generalized Schur map since $\xi_t$ and $I-\xi_t\circ\Lambda$ are generalized Schur maps. Observe that $\rho \circ \Psi_t = \eta_t(\rho)$
for all $\rho \in M_n(\C)^*$.
Since $\Psi_t$ is a generalized Schur map with respect to the decompositions
$\bigoplus_{i=1}^n \C$ and $\bigoplus_{i=1}^n B(L^2(0, \infty))$ for every $t>0$,
we have
\begin{equation}\label{N}
\Psi_t(X^{kk}) = \rho_0\Big(\Psi_t(X^{kk})\Big)e_{kk}  = \eta_t(\rho_0)(X^{kk})
e_{kk}
\end{equation}
for all $X \in B(L^2(0, \infty))$.

We define the linear functional $\omega$ acting on  $\mathfrak{A}(L^2(0, \infty))$ by
$$\omega(X) = \eta(\rho_0)(X^{kk})$$ for all $X \in \mathfrak{A}(L^2(0, \infty))$. It
follows from the discussion preceding the current Lemma that $\omega$ is indeed a well-defined
positive boundary weight in $\mathfrak{A}(L^2(0, \infty))_*$.

Denote the right shift semigroups acting on $B(H)$ and $B(L^2(0, \infty))$
by
$\{S_t\}_{t \geq 0}$ and $\{V_t\}_{t \geq 0}$, respectively.
Note that under our identification of $B(H)$ with $M_n(B(L^2(0, \infty)))$,  we have
$[S_t]_{kk}=V_t$ for every $t \geq 0$.
Define the truncated weights $\omega_t \in B(L^2(0, \infty))_*$ by
$$\omega_t(X) = \omega(V_tV_t^* X V_tV_t^*)$$ for $t>0$ and $X \in B(L^2(0, \infty))$. Then we
have that
\begin{align*}
\omega_t(X) & = \omega(V_t V_t^* X V_tV_t^*) =
\eta(\rho_0) \Big([V_tV_t^*XV_tV_t^*]^{kk}) \Big)
=  \eta(\rho_0)\Big(S_tS_t^* X^{kk} S_t S_t^*\Big) \\
& = \eta_t(\rho_0)(X^{kk}),
\end{align*}
whereby equation~\eqref{N} implies
$\Psi_t(X^{kk})= \omega_t(X)e_{kk}$ for all $X \in B(L^2(0, \infty))$.  Thus we have
$(I + \Psi_t \circ \Lambda)^{-1}(e_{kk}) = \dfrac{1}{1+ \omega_t(\Lambda)}e_{kk}.$  Therefore,
for all $X \in B(L^2(0, \infty))$,
\begin{eqnarray*}  \xi_t(X^{kk}) & = &  (I + \Psi_t \circ
\Lambda)^{-1}\Psi_t\Big(X^{kk}\Big)
= (I + \Psi_t \circ \Lambda)^{-1}\Big(\omega_t(X) e_{kk}\Big)
\\  & = &
\frac{\omega_t(X)}{1+\omega_t(\Lambda)} e_{kk}. \end{eqnarray*}
\end{proof}

\begin{thm}\label{invertiblecocycleconjugacy}
Let $\psi: M_n(\C) \rightarrow M_n(\C)$ be a unital invertible
$q$-pure map, and let $\nu$ be a type II Powers weight.  Then $(\psi, \nu)$
and $(\imath_\C, \nu)$ induce cocycle conjugate minimal dilation $E_0$-semigroups.
\end{thm}

\begin{proof}  By Theorem~\ref{inverts}, since
$\psi$ is a unital invertible $q$-pure map, it is conjugate to
a unital invertible $q$-pure  map $\phi$ of the form
\begin{equation*} [\phi(A)]_{jk} = \left\{
\begin{array}{cc}
\dfrac{a_{jk}}{1+i(\lambda_j - \lambda_k)} & \textrm{if } j<k \\ \noalign{\medskip}
a_{jk} & \textrm{if } j=k \\ \noalign{\medskip}
\dfrac{a_{jk}}{1-i(\lambda_j - \lambda_k)} & \textrm{if } j>k
\end{array} \right.
\end{equation*}
for all $j,k=1, \ldots, n$ and all $A=(a_{jk}) \in
M_n(\C)$, where $\lambda_1, \ldots, \lambda_n \in \R$ and $\lambda_1
+ \ldots + \lambda_n = 0$. By Proposition~\ref{arrgh}, $(\phi, \nu)$ and $(\psi, \nu)$ induce cocycle conjugate E$_0$-semigroups, therefore it suffices to show that the E$_0$-semigroups induced by $(\phi, \nu)$ and $(\imath_\C, \nu)$ are cocycle conjugate.

Define $\gamma: M_{n , 1}(\C)
\rightarrow M_{n , 1}(\C)$ by

\begin{displaymath} \gamma \left(
\begin{array}{c}
b_1 \\ b_2 \\ \vdots \\ b_n \end{array} \right) = \left(
\begin{array}{c}
\frac{1}{1+i\lambda_1 } b_1 \\
\frac{1}{1+i\lambda_2} b_2 \\
\vdots \\
\frac{1}{1+i\lambda_n} b_n \end{array} \right),
\end{displaymath}
and define a unital map $\Upsilon: M_{n+1}(\C) \rightarrow M_{n+1}(\C)$ by
\begin{equation}\label{upsilon-schur}
\Upsilon \left(
\begin{array}{ccc}
A_{n, n} & B_{n, 1} \\
C_{1, n} & d \end{array} \right) = \left(
\begin{array}{ccc}
\phi(A_{n, n}) & \gamma(B_{n, 1}) \\
\gamma^*(C_{1, n}) & d
\end{array} \right).
\end{equation}
Letting $\lambda_{n+1}=0$, we see that for all $A=(a_{ij}) \in M_{n+1}(\C)$,
\begin{equation*} [\Upsilon(A)]_{jk} = \left\{
\begin{array}{cc}
\dfrac{a_{jk}}{1+i(\lambda_j - \lambda_k)} & \textrm{if } j<k \\ \noalign{\medskip}
a_{jk} & \textrm{if } j=k \\ \noalign{\medskip}
\dfrac{a_{jk}}{1-i(\lambda_j - \lambda_k)} & \textrm{if } j>k
\end{array} \right.,
\end{equation*}
where $\lambda_1, \ldots, \lambda_{n+1} \in \R$ and $\sum_{k=1}^n \lambda_k = \sum_{k=1}^{n+1}
\lambda_k = 0$, so $\Upsilon$
is $q$-positive by Theorem~\ref{inverts}. By Proposition \ref{bdryweight}, the boundary
weight double $(\Upsilon, \nu)$ gives rise to a unital
CP-flow $\Theta=\{\Theta_t\}_{t \geq 0}$ over $\C^{n+1}$.

Let $\alpha$ and $\beta$ be the unital CP-flows over
$\C^n$ and $\C$, respectively, induced by $(\phi, \nu)$ and
$(\imath_\C, \nu)$.  Since $\Upsilon$ is a generalized Schur map in the sense of \eqref{upsilon-schur},
it follows from Remark \ref{summary} applied to its boundary weight map that $\Theta$ has the form
\begin{displaymath}
\Theta_t = \left(\begin{array}{cc} \alpha_t & \sigma_t \\
\sigma_t^* & \beta_t \end{array} \right)
\end{displaymath}
for some semigroup $\sigma=\{\sigma_t\}_{t \geq 0}$ of maps from $B(L^2(0, \infty), \C^n \otimes
L^2(0, \infty))$ into itself.  Suppose
\begin{displaymath}
\Theta \geq \Theta'  = \left(\begin{array}{cc} \alpha' & \sigma \\
\sigma^* & \beta' \end{array} \right)
\end{displaymath}
for some CP-flow $\Theta'$ over $\C^{n+1}$.
Let $\pi=\{\pi_t\}_{t>0}$ and $\pi' = \{\pi_t'\}_{t>0}$ be the
generalized boundary representations for $\Theta$ and $\Theta'$,
respectively.  Since $\Theta \geq \Theta'$ and $\Upsilon$ is a unital
$q$-positive Schur map, Lemma~\ref{diags} implies that for each $k=1, \ldots, n+1$,
there is some positive boundary weight $\omega_k \in \mathfrak{A}
(L^2(0, \infty))_*$ such that
\begin{equation} \label{pit'} \pi'_t(X^{kk}) =
\frac{(\omega_k)_t(X)}{1+(\omega_k)_t(\Lambda)} e_{kk} \end{equation}
for every $X \in B(L^2(0, \infty))$.
We note that $\phi$ is a Schur map, hence a direct calculation shows that
for $r \leq s \in \{1,\dots,n+1\}$,
\begin{equation} \label{pit}
\pi_t(X^{rs}) = \frac{\nu_t(X)}{1+\nu_t(\Lambda) + i (\lambda_r - \lambda_s)}e_{rs}.
\end{equation}
Since $\pi_t - \pi_t'$ is completely positive for all $t>0$, equations \eqref{pit'}
and \eqref{pit} (when $r=s=k$) imply that
$$\frac{\nu_t}{1+\nu_t(\Lambda)} - \frac{(\omega_k)_t}{1+(\omega_k)_t(\Lambda)}$$
is a positive functional in $B(L^2(0, \infty))_*$ for all $t>0$.  In other words, for $k=1, \dots, n+1$,
\begin{equation}\label{dom2}
\nu \geq_q \omega_k.
\end{equation}

Now let us fix $k \in \{1, \ldots, n\}$, and let
$\iota: M_2(B(L^2(0,\infty)) \to  M_{n+1}(B(L^2(0,\infty)))$ be the injective $*$-homomorphism given by
$$
\iota
\begin{pmatrix}
A_{11} & A_{12} \\
A_{21} & A_{22}
\end{pmatrix}
 = (A_{11})^{kk} + (A_{12})^{k,n+1} +
(A_{21})^{n+1,k}  + (A_{22})^{n+1,n+1}.
$$
Since $\iota$ is a $*$-homomorphism, it is clear that it is completely positive.
Let also $E:M_{n+1}(B(L^2(0,\infty))) \to M_2(B(L^2(0,\infty))$ be the completely positive map given by
$$
E( A ) = \begin{pmatrix}
A_{kk} & A_{k,n+1} \\
A_{n+1,k} & A_{n+1,n+1}
\end{pmatrix} .
$$

Now note that $\vartheta_t = E\circ \Theta_t \circ \iota$ and $\vartheta_t' = E\circ \Theta_t' \circ \iota$ are generalized Schur maps on $B(L^2(0,\infty)\oplus L^2(0,\infty))$ with respect to the decomposition $L^2(0,\infty) \oplus L^2(0,\infty)$. Furthermore, $\vartheta$ and $\vartheta'$ are CP-flows over $\cc\oplus\cc$. Let $\xi_t$ and $\xi_t'$ be their generalized boundary representations.
Now using the notation for generalized Schur maps, notice that since $\Theta$ and $\Theta'$ have the corner $\sigma$ in common, it follows that $[\vartheta_t]_{12}=[\vartheta_t']_{12}$.
Thus it follows by Proposition \ref{corner-computing} that $[\xi_t(X)]_{12}=[\xi_t'(X)]_{12}$
for all $X \in M_2(B(L^2(0, \infty)))$,
where $[\xi_t(X)]_{ij}=(\xi_t)_{ij}(X_{ij})$ and $[\xi_t'(X)]_{ij}=(\xi_t')_{ij}(X_{ij})$ for all $i,j=1,2$ since $\xi_t$
and $\xi_t'$ are generalized Schur maps.   Furthermore, observe that for every $t>0$ and $X= (X_{ij}) \in M_2(B(L^2(0, \infty)))$,
\begin{align*}
\pi_t( \iota(X) ) & = (\xi_t)_{11}(X_{11}) e_{kk} + (\xi_t)_{12}(X_{12})  e_{k,n+1} +
(\xi_t)_{21}(X_{21}) e_{n+1,k}  + (\xi_t)_{22}(X_{22}) e_{n+1,n+1}, \\
\pi'_t( \iota (X) ) & = (\xi_t')_{11}(X_{11}) e_{kk} + (\xi'_t)_{12}(X_{12})  e_{k,n+1} +
(\xi'_t)_{21}(X_{21}) e_{n+1,k}  + (\xi_t')_{22}(X_{22}) e_{n+1,n+1} \\
& = (\xi_t')_{11}(X_{11}) e_{kk} + (\xi_t)_{12}(X_{12})  e_{k,n+1} +
(\xi_t)_{21}(X_{21}) e_{n+1,k}  + (\xi_t')_{22}(X_{22}) e_{n+1,n+1}.
\end{align*}
Thus, by combining equations \eqref{pit'} and \eqref{pit} with the fact that $\lambda_{n+1}=0$, we obtain
that for every $X = (X_{ij}) \in  M_2(B(L^2(0,\infty)))$,
$$
\xi_t'(X) =
\begin{pmatrix}
\dfrac{(\omega_k)_t(X_{11})}{1+(\omega_k)_t(\Lambda)}
& \dfrac{\nu_t(X_{12})}{1+\nu_t(\Lambda) + i \lambda_k}  \\  \noalign{\medskip}
\dfrac{\nu_t^*(X_{21})}{1+\nu_t^*(\Lambda)
- i \lambda_k}  & \dfrac{(\omega_{n+1})_t(X_{22})}{1+(\omega_{n+1})_t(\Lambda)}
\end{pmatrix}.
$$
The above map is completely positive by construction for every $t>0$, since it is the
generalized boundary representation of a CP-flow. Hence
$\frac{1}{1+i \lambda_k} \nu$ is a $q$-corner from $\omega_k$ to $\omega_{n+1}$.
We know from Remark~\ref{hypnote} that $\frac{1}{1+i \lambda_k} \nu$ is a hyper-maximal $q$-corner
from $\nu$ to $\nu$,
so $\nu= \omega_k = \omega_{n+1}$ by \eqref{dom2}.

Thus we conclude that $\omega_k = \nu$ for each $k=1,
\ldots, n+1$, and it follows by \eqref{pit'} and \eqref{pit} that
$$
(\pi_t-\pi_t')(I) = 0.
$$
But $\pi_t - \pi_t'$ is completely positive by Theorem~\ref{boundary-representation-subordinates}, hence we have that
$$||\pi_t-\pi_t'|| = ||(\pi_t-\pi_t')(I)|| =0, $$ thus $\pi_t=\pi_t'$ for all $t>0$,   whereby
$\Theta'=\Theta$ again by Theorem~\ref{boundary-representation-subordinates}. Therefore,
$\sigma$ is a hyper-maximal flow corner from $\alpha$ to $\beta$.  Hence,
$\alpha^d$ and $\beta^d$ are cocycle conjugate by Theorem~\ref{thm-corners}.
\end{proof}

\section{Unitary equivalence of boundary weight maps}
The following proposition is a direct consequence of Bhat's theorem and Arveson's characterization
of minimality. Although we could not find a convenient reference for the it in the literature, we
believe that it is already known. We include a proof here for the convenience of the reader. We
thank Bob Powers for pointing out its role in sharpening the result which follows the proposition.

\begin{prop}\label{conjugacy-CP-semigroups}
Let $\alpha$ and $\beta$ be unital CP-semigroups acting on $B(K_\alpha)$, $B(K_\beta)$ with minimal
dilations $\alpha^d$ and $\beta^d$, respectively. Suppose that there exists a unitary $V:K_\beta \to
K_\alpha$  such that
$$
\beta_t(A) = V^*\alpha_t(VAV^*)V
$$
for all $A\in B(K_\beta)$ and $t\geq 0$. Then $\alpha^d$ and $\beta^d$ are conjugate
$E_0$-semigroups.
\end{prop}
\begin{proof} Suppose that $\alpha^d$ acting on $B(H)$ is a minimal dilation, i.e. there exists an
isometry $W:K_\alpha \to H$ such that $WW^*$ is an increasing projection for $\alpha^d$ for which
$$
\alpha_t(A)=W^*\alpha_t ^d(WAW^*)W
$$
for all $A\in B(K_\alpha)$ and $t\geq 0$ and furthermore $H = \cspan(\mathcal{S}_\alpha)$ where
$$
\mathcal{S}_\alpha=\{\alpha_{t_1} ^d(WA_1W^*) \cdots \alpha_{t_n} ^d(WA_nW^*) Wf
: f \in H, A_i \in B(K_\alpha), t_i \geq 0, n \in \mathbb{N}\}.
$$
In order to show that $\alpha^d$ and $\beta^d$ are conjugate, it suffices to show that $\alpha^d$ is
a minimal dilation of $\beta$.  This is equivalent
to showing that there exists an isometry 
$Z: K_\beta \to H$ such that $ZZ^*$
is an increasing projection for $\alpha^d$ for which
$$
\beta_t(A)=Z^*\alpha_t ^d(ZAZ^*)Z
$$
for all $A\in B(K_\beta)$ and $t\geq 0$ and furthermore $H = \cspan(\mathcal{S}_\beta)$
where
$$
\mathcal{S}_\beta=\{\alpha_{t_1} ^d(ZA_1Z^*) \cdots \alpha_{t_n} ^d(ZA_nZ^*) Zf
: f \in H, A_i \in B(K_\beta), t_i \geq 0, n \in \mathbb{N}\}.$$

Let $Z: K_\beta\rightarrow H$ be the isometry $Z=WV$. Note that
\begin{align*}
\beta_t(A) &= V^* \alpha_t(VAV^*)V = V^* \Big(W^* \alpha_t ^d(W(VAV^*)W^*) W\Big) V
\\ & =   V^*W^* \alpha_t ^d(WVAV^*W^*)WV = Z^* \alpha_t ^d(ZAZ^*)Z
\end{align*}
for all $A \in B(H)$ and $t \geq 0$. Furthermore, $ZZ^*$ is increasing for $\alpha^d$
because $ZZ^*=WW^*$ and $WW^*$ is increasing for $\alpha^d$.

Let $\xi \in \mathcal{S}_\alpha$, so that there exist $\{A_i\}_{i=1}^n \subset B(K_\alpha)$,
$t_i\geq 0$ for $i=1,\dots, n$, and $f \in K_\alpha$  such that
$$
\xi =  \alpha_{t_1} ^d(WA_1W^*) \cdots \alpha_{t_n} ^d(WA_nW^*)f.
$$
Letting $g=V^*f \in K_\beta$ and $B_i = V^*A_iV \in B(K_\beta)$ for all $i=1,\ldots n$, we observe
that
$$
\xi= \alpha_{t_1} ^d(WA_1W^*) \cdots \alpha_{t_n} ^d(WA_nW^*)f
= \alpha_{t_1} ^d(ZB_1Z^*) \cdots \alpha_{t_n} ^d(ZB_nZ^*)Zg,
$$
hence $\xi \in \mathcal{S}_\beta$. Therefore $\mathcal{S}_\alpha \subseteq \mathcal{S}_\beta$.
Consequently
$H=\cspan(S_\beta)$.
\end{proof}

Given a Hilbert space $H$ and a unitary $V\in B(H)$ we denote by $\Ad_V : B(H) \to B(H)$ the map
given by $\Ad_V(X)=V^*XV$. Thus we obtain the map $\wad_V : B(H)_* \to B(H)_*$ given by
$$
\wad_V(\rho)(X) = \rho(\Ad_V(X)) = \rho(V^*XV)
$$
for all $X \in B(H)$ and $\rho \in B(H)_*$.

\begin{thm}\label{conjugation-thm}
Let $H = K \otimes L^2(0, \infty)$, where $K$ is a separable Hilbert space.
Let $\alpha$ be a CP-flow over $K$,
and let $\omega: B(K)_* \to \mathfrak{A}(H)_*$ be its boundary weight map.  For every unitary $U \in
B(K)$, let $\U=U \otimes I_{L^2(0, \infty)} \in B(H)$. Define a map $\omega^U: B(K)_* \to \mathfrak{A}(H)_*$
by
$$
\omega^U  =  \wad_{\U^*} \circ \omega \circ \wad_U.
$$
Then $\omega^U$ is the boundary weight map of the unital CP-flow $\alpha_U$ over $K$ given by
\begin{equation}\label{alphu}
(\alpha_U)_t(A) = \U^* \alpha_t( \U A \U^*) \U \end{equation} for all $A \in B(H)$, $t \geq 0$.
\end{thm}
\begin{proof}  For each $t>0$, let
$E_{(t, \infty)}= S_t S_t^* \in B(H)$.  In other words, $E_{(t, \infty)} = I_K \otimes V_tV_t^*$,
where $V_t$ is the right shift by $t$ units on $L^2(0, \infty)$.
Note that $\Ad_{\U^*}$ leaves $\mathfrak{A}(H)$ invariant, so $\rho \rightarrow \omega^U(\rho)$
maps $B(K)_*$ into $\mathfrak{A}(H)_*$, therefore $\omega^U$
is well-defined. Furthermore, $\omega^U$
is the composition of completely positive maps and is therefore completely positive.  Note
that $\U$ commutes with $E_{(t, \infty)}$ for all $t>0$, so
\begin{eqnarray*}
\omega^U_t(\rho)(A) & = & \omega^U(\rho)( E_{(t, \infty)} A E_{(t, \infty)})=
\omega(\wad_{\U}(\rho))(\U E_{(t, \infty)} A E_{(t, \infty)}
\U^*) \\ & = & \omega(\wad_U(\rho))(E_{(t, \infty)} \U A \U^* E_{(t, \infty)})
= \omega_t(\wad_U(\rho)(\U A \U^*) \\
& = & (\wad_{\U^*} \circ \omega_t \circ \wad_U)(\rho),\end{eqnarray*} hence $\omega^U
_t(\rho) \in B(H)_*$
for all $\rho \in B(K)_*$ and all $t>0$.  Furthermore, we have
$$\omega^U _t(I + \widehat{\Lambda}\omega^U _t)^{-1} = \wad_{\U^*} \circ \omega_t(I +
\widehat{\Lambda}\omega_t)^{-1} \circ \wad_U,$$ so the maps $\widehat{\pi}_t: =
\omega^U _t(I + \widehat{\Lambda}\omega^U _t)^{-1}$ are completely positive contractions
of $B(K)_*$ into $B(H)_*$ for all $t>0$, hence $\omega^U$ is the boundary
weight map of a CP-flow $\alpha_U$ over $K$.  It remains to show that $\alpha_U$
is given by equation \eqref{alphu}.

Recall that the resolvent $R_{\alpha_U}$ for
$\alpha_U$ satisfies
\begin{equation} \label{resolve} \hat{R}_{\alpha_U}(\eta) (A)
= \int_0 ^ \infty e^{-t} (\alpha_U)_t(A) dt \qquad \textrm{ and } \qquad
\hat{R}_{\alpha_U}(\eta) = \hat{\Gamma} (\omega^U(\hat{\Lambda}\eta) + \eta)  \end{equation}
for all $A \in B(H)$, $\eta \in B(H)_*$.

We make four observations:
\begin{enumerate}
\item[(I)]  $\U S_t = S_t \U$ for all $t \geq 0$,\\ 
\item[(II)] $\U^* \Lambda(X) \U = \Lambda(U^*XU)$ for all $X \in B(K)$,\\ 
\item[(III)] $\wad_U\Big(\hat{\Lambda} \eta \Big) =
\hat{\Lambda}\Big(\wad_{\U}(\eta)\Big)$ for all $\eta \in B(H)_*$,\\
\item[(IV)]  $\U^* \Gamma(B)\U = \Gamma(\U^* B \U)$ for all $B \in B(H)$.
\end{enumerate}
Equation (I) and the fact that $\alpha$ is a CP flow imply that the mappings $A \rightarrow
\U^* \alpha_t(\U A \U^*) \U$ for $A \in B(H)$ and $t \geq 0$ define a CP-flow over $\C^n$, since
\begin{equation*}
\U^* \alpha_t(\U A \U^*) \U S_t = \U^* \Big( \alpha_t(\U A \U^*)  S_t \Big) \U
= \U^*(S_t \U A \U^*) \U = \U^* S_t \U A = S_t A.
\end{equation*} For all $\eta \in B(H)_*$ and $A \in B(H)$, we find:
\begin{eqnarray*}
\hat{\Gamma}\Big(\omega^U (\hat{\Lambda}\eta)\Big)(A) & = &
\omega\Big(\wad_U(\hat{\Lambda}\eta)\Big)(\U\Gamma(A)\U^*)
\\ \textrm{(by (III), (IV))} & = &
\omega\Big(\hat{\Lambda}\Big(\wad_{\U}(\eta)\Big)\Big)(\Gamma(\U A \U^*))
\\ \textrm{(by \eqref{resolvent})} & = &
\hat{R}_\alpha\Big(\wad_{\U}(\eta)\Big)(\U A \U^*)
- \wad_{\U}(\eta)( \Gamma(\U A \U^*))
\\ \textrm{(by \eqref{resolvent-pure}, (IV))} & = &  \wad_{\U}(\eta)\Big( \int_0 ^ \infty
e^{-t} \alpha_t
(\U A \U^*) dt \Big) - \eta(\U^* \Gamma(\U A \U^*) \U)
\\ \textrm{(by \ (IV))} &= & \eta \Big( \int_0 ^ \infty e^{-t} \U^* \alpha_t(\U A \U^*) \U
dt\Big)
- \eta(\Gamma(A)).
\end{eqnarray*}
The above equation and equation \eqref{resolve} give us
$$\eta \Big( \int_0 ^ \infty e^{-t} \U^* \alpha_t(\U A \U^*) \U dt\Big) =
\Big(\hat{\Gamma}(\omega^U(\hat{\Lambda}\eta) + \eta)
\Big)(A)=
\eta\Big(R_{\alpha_U}(A)\Big)$$
for all $\eta \in B(H)_*$ and $A \in B(H)$, hence $\alpha_U$ has the form of equation
\eqref{alphu}.
\end{proof}

\begin{cor}\label{uniflow}
If $\phi, \psi: M_n(\C) \rightarrow M_n(\C)$ are conjugate unital $q$-positive maps and $\nu$ is a
type II Powers weight, then the boundary weight doubles $(\phi, \nu)$ and $(\psi, \nu)$ have
conjugate minimal flow dilations.
\end{cor}
\begin{proof} Let $\nu$ be a type II Powers weight, and $\phi, \psi: M_n(\C) \rightarrow M_n(\C)$ be
conjugate unital $q$-positive maps and let $U \in M_n(\C)$ be a unitary such that
$\psi=\phi_U$. We remark that
$$
\widehat{\Omega}_\nu \circ \wad_{U^*}  =
\wad_{\U^*} \circ \widehat{\Omega}_\nu
$$
since  for all $A\in \mathfrak{A}(L^2(0,\infty)), X\in M_n(\C)$,
\begin{align*}
U\Omega_\nu(X \otimes A)U^* & = U(\nu(A)X)U^* = \nu(A)(UXU^*)  =
\Omega_\nu((UXU^*) \otimes A)\\
& = \Omega_\nu( \U (X \otimes A) \U^*).
\end{align*}
Therefore, if $\omega$ is the boundary weight map associated with the boundary weight double $(\phi,
\nu)$, i.e.
$
\omega = \widehat{\Omega}_\nu \circ \widehat{\phi}
$,
it follows that
$$
\omega^U  = \widehat{\Ad}_{\widetilde{U}^*} \circ \omega \circ \wad_U =
\wad_{\U^*} \circ \widehat{\Omega}_\nu \circ \widehat{\phi} \circ \wad_U
=  \widehat{\Omega}_\nu \circ \wad_{U^*} \circ \widehat{\phi} \circ \wad_U
=\widehat{\Omega}_\nu \circ \widehat{\phi}_U
$$
Thus, $\omega^U$ is the boundary weight map for the boundary weight double $(\phi_U, \nu)$.
Therefore, by Theorem~\ref{conjugation-thm} the unital CP-flow induced by $(\phi_U, \nu)=(\psi,
\nu)$ is conjugate to the unital CP-flow induced by $(\phi, \nu)$. Thus it follows from
Proposition~\ref{conjugacy-CP-semigroups} that the two unital CP-flows have conjugate minimal flow
dilations.
\end{proof}

We obtain as a consequence the following result.

\begin{thm}\label{rankoneconjugacy}  Let $\phi$ and $\psi$ be unital rank one $q$-positive maps on
$M_n(\C)$ and $M_k(\C)$, respectively, and let $\nu$ be a type II Powers
weight of the form
$$\pure.$$  Let $\alpha^d$
and $\beta ^d$ be the $E_0$-semigroups induced by $(\phi, \nu)$ and
$(\psi, \nu)$, respectively.  The following are equivalent:
\begin{enumerate}[(i)]
\item $\alpha^d$
and $\beta^d$ are conjugate.
\item $\alpha^d$
and $\beta^d$ are cocycle conjugate.
\item $n=k$ and $\phi$ is conjugate to $\psi$.
\end{enumerate}
\end{thm}
\begin{proof}  Trivially, (i) implies (ii), while (ii) implies (iii) by Theorem 3.10 of
\cite{jankowski2}.
Corollary~\ref{uniflow} shows that (iii) implies (i). \end{proof}

\section{Gauge group in the range rank one case}
In this section, we will calculate the gauge group for the minimal flow dilation $\alpha^d$ of the CP-flow
$\alpha$ induced by the boundary
weight double $(\phi, \nu)$, where $\phi: M_n(\C) \rightarrow M_n(\C)$ is a unital rank one $q$-positive map
and $\nu$ is a type II Powers weight of the form
$$\pure.$$

In the context of CP-flows, the local unitary cocycles are more conveniently described in terms of the associated hyper-maximal flow corners. This description remains out of reach in the case of general boundary weight doubles, however in the special case when  $\nu$ has the form $\pure$ and $\phi: M_n(\C) \rightarrow M_n(\C)$ is \emph{any} unital $q$-positive map, we present a convenient description. In the following theorem we describe explicitly  a one-to-one correspondence between the hyper-maximal flow corners from $\alpha$ to $\alpha$ and the hyper-maximal $q$-corners from $\phi$ to $\phi$.

\begin{thm}\label{corner-correspondence}
Let $\phi: M_n(\C) \rightarrow M_n(\C)$ be a unital $q$-positive map, and let $\nu$ be a type II
Powers
weight of the form $$\pure.$$  Let $\alpha$ be the unital CP-flow induced by the boundary weight
double
$(\phi, \nu)$.

Suppose $\gamma$ is a hyper-maximal $q$-corner from $\phi$ to $\phi$.   Define a
linear map $\omega: M_{2n}(\C)_* \rightarrow \fa(\C^{2n} \otimes L^2(0, \infty))_*$ by
\begin{equation} \label{boundary-weight-corner}
\omega(\rho)
\begin{pmatrix}
A_{11} & A_{12}
\\ A_{21} & A_{22}
\end{pmatrix}
=\rho \begin{pmatrix}
\phi(\Omega_\nu(A_{11})) & \gamma(\Omega_\nu(A_{12}))
\\ \gamma^*(\Omega_\nu(A_{21})) & \phi(\Omega_\nu(A_{22}))
\end{pmatrix}.
\end{equation}
Then $\omega$ is the boundary weight map of a unital CP-flow $\Theta$ of the form
$$
\Theta = \begin{pmatrix}  \alpha & \sigma \\ \sigma^* & \alpha \end{pmatrix},
$$
where $\sigma$ is a hyper-maximal flow corner from $\alpha$ to $\alpha$.  The generalized
boundary representation $\Pi_t$ for $\Theta$ is given by
\begin{equation}\label{genrep4.6}
\Pi_t = \begin{pmatrix}
\phi(I + \nu_t(\Lambda) \phi)^{-1} \Omega_{\nu_t} & \gamma(I + \nu_t(\Lambda) \gamma)^{-1} \Omega_{\nu_t}
\\ \gamma^*(I + \nu_t(\Lambda) \gamma^*)^{-1} \Omega_{\nu_t} & \phi(I + \nu_t(\Lambda) \phi)^{-1} \Omega_{\nu_t}
\end{pmatrix}
\end{equation}
for all $t>0$.

Conversely, suppose that $\sigma$ is a hyper-maximal flow corner from $\alpha$ to $\alpha$.  Let $\Theta$
be the CP-flow
$$
\Theta = \begin{pmatrix}  \alpha & \sigma \\ \sigma^* & \alpha \end{pmatrix}.
$$
Let $\omega$ be the
boundary weight map
for $\Theta$ and let $\Pi_t$ be the generalized  boundary representation for $\Theta$.  Then there exists a unique hyper-maximal $q$-corner
$\gamma$ from $\phi$ to $\phi$ such that $\omega$ is given by equation \eqref{boundary-weight-corner}.
Furthermore, $\Pi_t$ satisfies equation \eqref{genrep4.6} for every $t>0$.
\end{thm}

\begin{proof} We will use two key facts established in the proof of Proposition 4.6 of \cite{jankowski1}.
Suppose that $\gamma$ is a hyper-maximal $q$-corner from
$\phi$ to $\phi$. 
It was shown in the proof of Proposition 4.6 of \cite{jankowski1} that the boundary
weight map defined by
\eqref{boundary-weight-corner}
induces a unital CP-flow of the form
\begin{displaymath}
\Theta= \left( \begin{array}{cc} \alpha & \sigma \\ \sigma^* & \alpha \end{array} \right),
\end{displaymath}
where $\sigma$ is a hyper-maximal flow corner from $\alpha$ to $\alpha$.  The fact that
the generalized boundary representation $\Pi_t$ for $\Theta$ satisfies \eqref{genrep4.6}
is a direct consequence of the formula
$\widehat{\Pi}_t = \omega_t(I + \hat{\Lambda}\omega_t)^{-1}$.  This proves the forward direction.

For the backward direction, let $\sigma$ be a hyper-maximal flow corner from $\alpha$ to $\alpha$,
and let $\Theta$ be the CP-flow
\begin{displaymath}
\Theta= \left( \begin{array}{cc} \alpha & \sigma \\ \sigma^* & \alpha \end{array} \right).
\end{displaymath}
Let $\omega$ be the boundary weight map and let $\Pi_t$ be the generalized boundary representation for $\Theta$.  In the proof of Proposition
4.6 of \cite{jankowski1}, it was shown that there exists a hyper-maximal $q$-corner
$\gamma$ from $\phi$ to $\phi$ such that $\Pi_t$ is given by \eqref{genrep4.6} for every $t>0$.

It remains to show that $\omega$ satisfies equation \eqref{boundary-weight-corner} and to establish that $\gamma$ is unique. For the former, observe that by Proposition~\ref{bdryweight}, since
$$
\begin{pmatrix} \phi & \gamma \\ \gamma^* & \phi
\end{pmatrix}
$$
is unital and $q$-positive, the boundary weight map $\omega'$ defined by
$$
\omega'(\rho)
\begin{pmatrix}
A_{11} & A_{12} \\ A_{21} & A_{22}
\end{pmatrix}
= \rho
\begin{pmatrix}
\phi(\Omega_\nu(A_{11})) & \gamma(\Omega_\nu(A_{12}))  \\ \gamma^*(\Omega_\nu(A_{21}))
& \phi (\Omega_\nu(A_{22}))
\end{pmatrix}
$$
induces a unital CP-flow $\Theta'$. By the forward direction of the theorem, its generalized boundary representation $\Pi'_t$ satisfies \eqref{genrep4.6}. Thus $\Pi_t=\Pi'_t$ for all $t>0$ and it follows that $\Theta=\Theta'$ and $\omega=\omega'$, establishing \eqref{boundary-weight-corner}.

We now show that $\gamma$ is unique. Suppose
$\gamma': M_n(\C) \rightarrow M_n(\C)$ is another linear map such that
$$
\omega(\rho)
\begin{pmatrix}
A_{11} & A_{12} \\ A_{21} & A_{22}
\end{pmatrix}
= \rho
\begin{pmatrix}
\phi(\Omega_\nu(A_{11})) & \gamma'(\Omega_\nu(A_{12}))  \\ (\gamma')^*(\Omega_\nu(A_{21}))
& \phi (\Omega_\nu(A_{22}))
\end{pmatrix}
$$
for all $\rho \in M_{2n}(\C)_*$ and $(A_{ij}) \in \fa(\C^{2n} \otimes L^2(0, \infty))$.
It follows that $\gamma' \circ \Omega_\nu = \gamma \circ \Omega_\nu$. Since $\Omega_\nu$ is onto, we conclude that $\gamma=\gamma'$.
\end{proof}
In light of the bijection between hyper-maximal flow corners from $\alpha$ to $\alpha$
and elements of $G_{flow}(\alpha^d)$ given by Theorem \ref{thm-corners}, we present
an immediate corollary of Theorem
\ref{corner-correspondence}.
\begin{cor}\label{bijection-between-corners}
Let $\phi: M_n(\C) \rightarrow M_n(\C)$ be a unital $q$-positive map, and let
$\nu$ be a type II Powers weight of the form
$$\pure.$$
Let $\alpha$ be the CP-flow induced by $(\phi, \nu)$.  Then there is a bijection
between hyper-maximal $q$-corners from $\phi$ to $\phi$ and elements of
$G_{flow}(\alpha^d)$.
\end{cor}

The following result is a combination of Theorems 3.8 and 3.9 of \cite{jankowski2}.

\begin{thm}\label{mainold}
Let $\{\mu_i\}_{i=1}^k$ and $\{r_i\}_{i=1}^{k'}$ be
non-increasing sequences of strictly positive numbers such that
$\sum_{i=1}^k \mu_k = \sum_{i=1}^{k'} r_i =1$. Define unital
$q$-positive maps $\phi: M_n(\C) \rightarrow M_n(\C)$ and $\psi:
M_{n'}(\C) \rightarrow M_{n'}(\C)$ (where $n \geq k$ and $n' \geq k'$) by
\begin{equation}\label{forms} \phi(A) =
\Big( \sum_{i=1}^k \mu_i a_{ii}\Big)I_n \ \ \textrm{ and }  \ \ \psi(D)
= \Big(\sum_{i=1}^{k'} r_i d_{ii}\Big)I_{n'}\end{equation} for all $A=(a_{ij}) \in
M_n(\C)$ and $D=(d_{ij}) \in M_{n'}(\C)$.  Let $\widetilde{\Omega} \in M_k(\C)$
be the diagonal matrix such that $\widetilde{\Omega}_{jj} = \mu_j$ for $j=1,\dots, k$.

If there is a nonzero $q$-corner from $\phi$ to $\psi$, then $k=k'$
and $\mu_i = r_i$ for all $i=1, \ldots, k$.  In that case, a linear
map $\gamma: M_{n,n'}(\C) \rightarrow M_{n,n'}(\C)$ is a $q$-corner
from $\phi$ to $\psi$ if and only if: for some unitary $V \in
M_{k}(\C)$ that commutes with $\widetilde{\Omega}$, some contraction $E \in
M_{n-k, n'-k}(\C)$, and some $\lambda \in \C$ with $|\lambda|^2 \leq
\re(\lambda)$, we have
\begin{displaymath}
\gamma \left(  \begin{array}{cc} B_{k,k} & W_{k, n'-k}
\\ Q_{n-k,k} & Y_{n-k, n'-k}     \end{array} \right)
= \lambda \ \tr(V^*  B_{k,k} \widetilde{\Omega})
\left(  \begin{array}{cc} V & 0_{k,n'-k} \\ 0_{n-k,k} & E
 \end{array} \right)
\end{displaymath}
for all
\begin{displaymath}
\left(  \begin{array}{cc} B_{k,k} & W_{k, n'-k}
\\ Q_{n-k,k} & Y_{n-k, n'-k}     \end{array} \right) \in M_{n,n'}(\C).
\end{displaymath}
A $q$-corner $\gamma: M_{n,n'}(\C)
\rightarrow M_{n,n'}(\C)$ from $\phi$ to $\psi$ is hyper-maximal if and only if
$n=n'$, $0<|\lambda|^2 = \re(\lambda)$, and $E$
is unitary.
\end{thm}

  Note that  $0 < |\lambda|^2 = \re(\lambda)$ if and only if $\lambda=\frac{1}{1+ix}$ where $x
\in \rr$.

\begin{defn}
Let $\rho \in M_n(\C)^*$ be a  state with trace density matrix $\Omega$. Let $U_\rho$ denote the set
of all unitaries $U\in M_n(\C)$ such that $U\Omega=\Omega U$, and let $U_\rho/\mathbb{T}$ denote the
group obtained by the identification $X \sim Y$ if and only if $X=cY$ where $c\in\C$ with $|c|=1$.
Let $G_\rho$ be the group $$G_\rho = \rr \times (U_\rho/\mathbb{T})$$ with the coordinate-wise product. Each
element $g \in G_\rho$ can be represented by a pair $(x,X) \in \rr \times U_\rho$, and we denote
this relationship by $g=\{x, X\}$. Using this notation,
$$
\{x, X\} \cdot \{y, Y \} = \{ x+y, XY \}.
$$

\end{defn}

  We record the following useful consequence of Theorem~\ref{mainold}.

\begin{thm}\label{rho-q-corners}
Let $\phi: M_n(\C) \rightarrow M_n(\C)$ be a rank one unital $q$-positive map, so $\phi(A)=\rho(A)
I$
for some state $\rho$ with trace density matrix $\Omega$.  Suppose that
$\{x,X\} \in G_\rho$. Then the map
$$
\gamma_{\{x, X\}} (A) = \frac{1}{1+ix} \tr(X^* A \Omega)X
$$
is a well-defined hyper-maximal $q$-corner from $\phi$ to $\phi$. Conversely, if $\gamma$ is
a hyper-maximal $q$-corner from $\phi$ to $\phi$, then there exists $\{x,X\}\in G_\rho$ such that
$\gamma = \gamma_{\{x, X\}}$.

Furthermore, if $g,h \in G_\rho$ and $\gamma_g=\gamma_h$, then $g=h$.
\end{thm}
\begin{proof}
We observe that if  $(x,X),(y,Y) \in \rr \times U_\rho$ are two representatives for an element of
$G_\rho$, then $x=y$ and $X=cY$ for $c\in\cc$ with $|c|=1$, hence $\gamma_{\{x, X\}}=\gamma_{\{y,
Y\}}$. Therefore this is a well-defined map parametrized by an element of $G_\rho$.

There exists a unitary $U \in M_n(\C)$ such that $\Omega_U=U^*\Omega U$ is diagonal
with nonzero non-increasing diagonal entries $\mu_1, \dots, \mu_k$ for some $k\leq n$. Let
us denote by $\rho_U(A)=\tr(A\Omega_U)$, and let $\phi_U$ be the $q$-positive unital map given by
$\phi_U(A)=\rho_U(A) I$.  By Proposition 4.5 of \cite{jankowski1} and
Remark 3.3 of \cite{jankowski2}, we
know that $\gamma$ is a hyper-maximal $q$-corner from $\phi$ to $\phi$ if and only if
$\gamma(A)= U \sigma(U^*AU) U^*$ where $\sigma$ is a hyper-maximal $q$-corner from $\phi_U$ to
$\phi_U$.

Let $\widetilde{\Omega} \in M_k(\C)$ be the diagonal matrix such that
$\widetilde{\Omega}_{jj}=\mu_j$ for $j=1, \dots, k$.
A straightforward calculation shows that a unitary matrix $Z \in M_n(\C)$ commutes with $\Omega_U$
if and only if it has the form
\begin{displaymath}
Z = \left( \begin{array}{cc} V & 0_{k, n-k}
\\ 0_{n-k, k} & E  \end{array}  \right),
\end{displaymath}
where $V \in M_k(\C)$ and $E \in M_{n-k}(\C)$ are unitary matrices and $V$ commutes with
$\widetilde{\Omega}$.
Furthermore, if
\begin{displaymath}
A= \left(  \begin{array}{cc} B_{k,k} & W_{k, n'-k}
\\ Q_{n-k,k} & Y_{n-k, n'-k}     \end{array} \right) \in M_n(\C),
\end{displaymath}
then $\tr(Z^* A \Omega_U) = \tr(V^* B \widetilde{\Omega})$.  It follows from Theorem~\ref{mainold}
that $\sigma$ is a hyper-maximal $q$-corner from $\phi_U$ to $\phi_U$
if and only if it has the form
$\sigma(A)= \frac{1}{1+ix} \tr(Z^* A \Omega_U) Z$ for all $A \in M_n(\C)$, where $x\in\rr$ and $Z
\in M_n(\C)$ is a unitary
matrix that commutes with $\Omega_U$.  Thus we have that
\begin{align*}
\gamma(A) & =  U \sigma(U^*AU) U^*  = \frac{1}{1+ix} \tr(Z^*(U^*AU)\Omega_U) UZU^*
= \frac{1}{1+ix} \tr(Z^*U^*A\Omega U) UZU^* \\
&  = \frac{1}{1+ix} \tr(X^*A\Omega) X = \gamma_{\{x,X\}}(A)
\end{align*}
where $X=UZU^*$. It is clear that $X$ commutes with $\Omega$ since $Z$ commutes with $\Omega_U$,
hence $(x,X)$ represents an element of $G_\rho$.

The uniqueness statement is clear, once one observes that since $\gamma_{\{x,X\}}(A)$ is always a
multiple of $X$, if  $\gamma_{\{x,X\}} = \gamma_{\{y,Y\}}$ then  $X$ and $Y$ are unitaries which
must be multiples of each other.
\end{proof}

Let $\phi(A)=\rho(A)I$ where $\rho \in M_n(\cc)^*$ is a state, and let $\nu$ be a type II Powers
weight of the form $\pure$. By Corollary \ref{bijection-between-corners}
and Theorem~\ref{rho-q-corners}, we
have a bijection relating each element $g \in G_\rho$ to a hyper-maximal $q$-corner $\gamma_g$ and
its corresponding  local unitary flow $\alpha^d$-cocycle which we denote by $C_g$.
Let $g,h \in G_\rho$.  Since the product of local unitary flow cocycles is also a local unitary flow cocycle,
it follows that $C_g \cdot C_h = C_s$ for some $s \in G_\rho$. We will prove that $s=gh$.  The
following simple lemma will prove useful in doing so.

\begin{lem}\label{special}
Let $\iota: M_3(\C) \to M_3(M_{n}(\C))$ be the natural inclusion given by $[\iota(A)]_{ij}=a_{ij}
I_n$.  Let $X,Y \in M_n(\C)$ be unitary, and let $V \in M_3(M_{n}(\C))$ be the unitary matrix given by
$$
V = \begin{pmatrix}
Y & 0 & 0 \\
0 & I & 0 \\
0 & 0 & XY
\end{pmatrix}.
$$
Then a linear map $L :M_3(M_n(\C)) \to M_3(\C)$ is completely positive if and only if the map $\phi: M_3(M_n(\C)) \to
M_3(M_{n}(\C))$ given by $\phi(A) = V\iota(L(A))V^*$ is completely positive.
\end{lem}

\begin{proof}  It is clear that $\phi$ is completely positive if and only if $\iota \circ L$ is
completely positive. On the other hand, $\iota$ is a *-isomorphism onto its range, therefore $\iota
\circ L$ is completely positive if and only if $L$ is completely positive.
\end{proof}

\begin{remark}\label{simple-lemma}
We will use the lemma in the special case when $L: M_3(M_n(\C)) \to M_3(\C)$ is of the form
$$
[L(A)]_{ij} = \ell_{ij}(A_{ij}),
$$
where $\ell_{ij} \in M_n(\C)^*$ for all $i,j=1, 2,3$. In this case,
$$
\phi(A) =  \left( \begin{array}{ccc} \ell_{11}
(A_{11}) I & \ell_{12}(A_{12}) Y & \ell_{13}(A_{13})X^* \\
\ell_{21}(A_{21}) Y^*
& \ell_{22}(A_{22}) I & \ell_{23}(A_{23}) Y^*X^* \\
\ell_{31}(A_{31})X & \ell_{32}(A_{32}) XY & \ell_{33}(A_{33}) I
\end{array} \right).
$$

\end{remark}

We will also use the following lemma (which appears in \cite{walter} and which is also a special case of Lemma 2.16 of \cite{alevras-powers-price}).

\begin{lem}\label{APP-2.16}
Suppose $K$ is a Hilbert space and $T \in M_3(B(K))$ has the form
$$T=
\begin{pmatrix}
I  & Y & X^*
\\ Y^* & I & Z^*
\\ X & Z & I
\end{pmatrix},
$$
where $X$ and $Y$ are unitary.  Then $T$ is positive if and only if $Z=XY$.
\end{lem}

We are now ready to prove the main result of the section.

\begin{thm}\label{gauge-groups-final}  Let $\nu$ be
a type II Powers weight of the form $$\pure,$$ and let $\phi: M_n(\C) \rightarrow M_n(\C)$ be a
unital
rank one $q$-positive map, so for some state $\rho$ we have
$\phi(A) = \rho(A)I$ for all $A \in M_n(\C)$. Let $\alpha^d$ be the minimal flow dilation
of the CP-flow $\alpha$ induced by the boundary weight double $(\phi, \nu)$. Then the map
$g \mapsto C_g$ is an isomorphism from $G_\rho$ onto $G_{flow}(\alpha^d)$, thus $G(\alpha^d) \simeq
\rr \times G_\rho$.
\end{thm}
\begin{proof} Recall that since $\alpha^d$ is type II$_0$, we have that $G(\alpha^d)$ is canonically isomorphic to $\rr \times
G_{flow}(\alpha^d)$ (for more details see the discussion following Definition~\ref{gflow}). Furthermore, it follows from Theorem~\ref{rho-q-corners} that the map $g \to
C_g$ described in the statement of the current theorem is well-defined, injective and surjective.
Thus, in order to complete the proof of the theorem it suffices to prove that it preserves
multiplication.

Let $\{x, X\}, \{y, Y\} \in G_\rho$ be given, and let $\Omega$ be the trace density matrix for
$\rho$. Then
$$
\gamma_{\{x,X\}}(A) = \frac{\tr(X^* A\Omega)}{1+ix}X, \qquad \gamma_{\{y,Y\}}(A) = \frac{\tr(Y^*A
\Omega)}{1+iy} Y
$$
for all $A \in M_n(\C)$.  For each $S \in M_n(\C)$, define $\tau_S\in M_n(\C)^*$ by $\tau_S(A) =
\tr(S\Omega^{1/2} A \Omega^{1/2})$, so that for example $\tau_I=\rho$.  Given $A \in M_{3n}(\C),$ we
write
$A=(A_{ij}) \in M_3(M_n(\C))$.  Let $\Theta: M_{3n}(\C) \rightarrow
M_{3n}(\C)$ be the map
\begin{displaymath}
\Theta \left( \begin{array}{ccc}  A_{11} & A_{12} & A_{13}\\
A_{21} & A_{22} & A_{23} \\ A_{31} & A_{32} & A_{33}
\end{array} \right)
=
\left( \begin{array}{ccc}  \phi(A_{11}) & \gamma_{\{y,Y\}} (A_{12}) & \gamma_{\{x,X\}}^*
(A_{13})\\
\gamma _{\{y,Y\}}^*(A_{21}) & \phi(A_{22}) & \gamma_{\{x+y, XY\}}^* (A_{23}) \\
\gamma_{\{x,X\}}(A_{31}) & \gamma_{\{x+y, XY\}} (A_{32}) &
\phi(A_{33})
\end{array} \right).
\end{displaymath}
For each $t \geq 0$ and $A \in M_{3n}(\C)$,
$$
\Theta(I +t \Theta)^{-1}(A) =
\begin{pmatrix}
\dfrac{\tau_I(A_{11})}{1+t} I  &
\dfrac{\tau_Y^* (A_{12})}{1+t+iy}  Y &
\dfrac{\tau_X (A_{13})}{1+t-ix} X^* \\ \noalign{\bigskip}
\dfrac{\tau_Y (A_{21})}{1+t-iy}Y^* & \dfrac{\tau_I(A_{22})}{1+t}I &
\dfrac{\tau_{XY} (A_{23})}{1+t-ix-iy}  Y^*X^* \\ \noalign{\bigskip}
\dfrac{\tau_X ^*(A_{31})}{1+t+ix}X & \dfrac{\tau_{XY}^*
(A_{32})}{1+t+ix+iy}XY & \dfrac{\tau_I(A_{33})}{1+t} I
\end{pmatrix}.
$$

By Lemma~\ref{special} and Remark~\ref{simple-lemma}, $\Theta$ is $q$-positive if and only if the
following maps $B_t: M_{3n}(\C) \rightarrow M_3(\C)$ are completely
positive for all $t \geq 0$:
$$
B_t(A)=
\begin{pmatrix}

\dfrac{\tau_I(A_{11})}{1+t}   &
\dfrac{\tau_Y^* (A_{12})}{1+t+iy}   &
\dfrac{\tau_X (A_{13})}{1+t-ix}  \\ \noalign{\bigskip}
\dfrac{\tau_Y (A_{21})}{1+t-iy} & \dfrac{\tau_I(A_{22})}{1+t} &
\dfrac{\tau_{XY} (A_{23})}{1+t-ix-iy}   \\ \noalign{\bigskip}
\dfrac{\tau_X ^*(A_{31})}{1+t+ix} & \dfrac{\tau_{XY}^*
(A_{32})}{1+t+ix+iy} & \dfrac{\tau_I(A_{33})}{1+t}
\end{pmatrix}.
$$
Let $Z=(Z_{ij}) \in M_3(M_n(\cc))$. We remark that if the  matrix $(Z_{ji})$ is positive (the
transposition of the indices is not a mistake), then the map $M_3(M_n(\cc))\to M_3(\cc)$ given by
$(A_{ij}) \mapsto (\tau_{Z_{ij}}(A_{ij}))$ is completely positive. For more details, see the
discussion preceding Lemma 2.17 in \cite{alevras-powers-price}. Thus $B_t$ is completely positive if

$$
M_t= \left(
\begin{array}{ccc} \frac{1}{1+t}I & \frac{1}{1+t-iy} Y &
\frac{1}{1+t+ix} X^*  \\
\frac{1}{1+t+iy}Y^*  & \frac{1}{1+t}I & \frac{1}{1+t+ix+iy} Y^*X^* \\
\frac{1}{1+t-ix}X & \frac{1}{1+t-ix-iy} XY & \frac{1}{1+t}I
\end{array} \right)
$$
is a  positive matrix. On the other hand, if
$$
V = \begin{pmatrix}
Y & 0 & 0 \\
0 & I & 0 \\
0 & 0 & XY
\end{pmatrix} \quad \text{and} \quad
N_t= \left(
\begin{array}{ccc} \frac{1}{1+t}  & \frac{1}{1+t-iy} &
\frac{1}{1+t+ix}   \\
\frac{1}{1+t+iy}  & \frac{1}{1+t} & \frac{1}{1+t+ix+iy} \\
\frac{1}{1+t-ix} & \frac{1}{1+t-ix-iy} & \frac{1}{1+t}
\end{array} \right),
$$
then $M_t=V^*\iota(N_t)V$ (where $\iota$ is the natural inclusion defined in Lemma~\ref{special}).
Therefore, $M_t$ is positive if and only $N_t$ is positive.

To see that $N_t$ is positive, let $\lambda_1 = \frac{x-y}{3}, \lambda_2 = \frac{x+2y}{3},$ and
$\lambda_3 = \frac{-2x-y}{3}$, and define
$\sigma: M_3(\C) \rightarrow M_3(\C)$ by
$$
[\sigma(A)]_{jk} = \left\{ \begin{array}{cc}
\dfrac{a_{jk}}{1+i(\lambda_j - \lambda_k)} & \textrm{if } j<k \\ \noalign{\medskip}
a_{jk} & \textrm{if } j=k \\  \noalign{\medskip}
\dfrac{a_{jk}}{1-i(\lambda_j - \lambda_k)} & \textrm{if } j>k
\end{array} \right.
$$
for all $A \in M_3(\C)$.  We have that $\sigma$ is $q$-positive by Theorem 6.11 of
\cite{jankowski1}, so
\begin{displaymath}
0 \leq \sigma(I + t\sigma)^{-1} \left( \begin{array}{ccc} 1 & 1 & 1\\ 1 & 1 & 1\\ 1
& 1 & 1 \end{array}
\right)
=N_t.
\end{displaymath}
This shows that $B_t$ is completely positive for all $t \geq 0$, so $\Theta$ is $q$-positive. Now
observe that $\Theta(I)=I$, so by Proposition~\ref{bdryweight}, the boundary weight
double $(\Theta, \nu)$
induces a unital CP-flow $\beta$ through the boundary weight map $\omega: B(\C^{3n})_*
\rightarrow \fa(\C^{3n} \otimes L^2(0, \infty))$ below:
$$\omega(\rho)(A)= \rho
\begin{pmatrix} \phi(\Omega_\nu (A_{11})) & \gamma_{\{y,Y\}} (\Omega_\nu(A_{12})) &
\gamma_{\{x,X\}}^* (\Omega_\nu(A_{13}))\\
\gamma _{\{y,Y\}}^*(\Omega_\nu(A_{21})) & \phi(\Omega_\nu(A_{22})) & \gamma_{\{x+y, XY\}}^*
(\Omega_\nu(A_{23})) \\
\gamma_{\{x,X\}}(\Omega_\nu(A_{31})) & \gamma_{\{x+y, XY\}} (\Omega_\nu(A_{32})) &
\phi(\Omega_\nu(A_{33}))
\end{pmatrix}.
$$
By Theorem~\ref{corner-correspondence}, each hyper-maximal $q$-corner from $\phi$ to $\phi$ is associated to a unique hyper-maximal flow corner from $\alpha$ to $\alpha$. Thus, each of the hyper-maximal $q$-corners $\gamma_{\{x,X\}}$, $\gamma_{\{y,Y\}}$, and $\gamma_{\{x+y,XY\}}$ is  uniquely associated to  its corresponding hyper-maximal flow corner from $\alpha$ to $\alpha$, which we denote by $\alpha_{\{x,X\}}$, $\alpha_{\{y,Y\}}$, and $\alpha_{\{x+y,XY\}}$.
By Proposition~\ref{submatrix-flow-corners}, $\beta$ is the positive $3
\times 3$ matrix of flow corners (Definition~\ref{def-corners}) given by
$$\beta=
\begin{pmatrix}
\alpha & \alpha_{\{y,Y\}} & \alpha_{\{x,X\}} ^* \\
\alpha_{\{y,Y\}} ^* & \alpha & \alpha_{\{x+y, XY\}} ^* \\
\alpha_{\{x,X\}} & \alpha_{\{x+y,XY\}} & \alpha
\end{pmatrix}.
$$

By Theorem~\ref{thm-corners}, $\beta$ corresponds to a unique positive $3 \times 3$ matrix
$C=(C_{ij})$ of contractive local flow $\alpha^d$-cocycles.  It follows from the
form of $\beta$ that
\begin{displaymath} C(t)=
\left( \begin{array}{ccc} I & C_{\{y,Y\}}(t) & C_{\{x,X\}}(t)^* \\
C_{\{y,Y\}}(t)^* & I & C_{\{x+y, XY\}}(t)^* \\ C_{\{x,X\}}(t) &
C_{\{x+y,XY\}}(t) & I  \end{array} \right)
\end{displaymath}
for all $t \geq 0$.  Since $C_{\{x,X\}}(t)$ and  $C_{\{y,Y\}}(t)$ are unitaries and
$C(t)$ is positive, it follows from Lemma~\ref{APP-2.16} that $C_{\{x+y,XY\}}(t)
= C_{\{x,X\}}(t) C_{\{y,Y\}}(t)$ for all $t \geq 0$.
\end{proof}

\providecommand{\bysame}{\leavevmode\hbox to3em{\hrulefill}\thinspace}
\providecommand{\MR}{\relax\ifhmode\unskip\space\fi MR }
\providecommand{\MRhref}[2]{%
  \href{http://www.ams.org/mathscinet-getitem?mr=#1}{#2}
}
\providecommand{\href}[2]{#2}

\end{document}